\documentclass[11pt,reqno]{amsart}
\usepackage[
  letterpaper,
  margin=1.15in
]{geometry}
\usepackage[pagewise]{lineno}
\usepackage{amsmath,amsthm,amssymb,amsfonts,relsize}
\usepackage[english]{babel}
\usepackage{graphicx}
\usepackage{enumitem}
\usepackage{bm}
\usepackage{float}
\usepackage{mathtools}
\usepackage{theoremref}
\usepackage{tikz}
\usetikzlibrary{arrows.meta}
\usepackage{pgfplots}
\usepgfplotslibrary{fillbetween}
\pgfplotsset{compat=1.18}
\numberwithin{equation}{section}
\newtheorem{theorem}{Theorem}[section]
\newtheorem{remark}[theorem]{Remark}
\newtheorem{corollary}[theorem]{Corollary}
\newtheorem{lemma}[theorem]{Lemma}
\newtheorem{proposition}[theorem]{Proposition}

\usepackage{xcolor}
\usetikzlibrary{arrows.meta,calc,positioning}
\mathtoolsset{showonlyrefs}
\definecolor{axiscolor}{HTML}{111827}
\definecolor{tealcurve}{HTML}{0F766E}
\definecolor{redcurve}{HTML}{B42318}
\definecolor{graytext}{HTML}{475569}
\definecolor{slate}{HTML}{334155}
\usepackage{hyperref}
\DeclarePairedDelimiter{\norm}{\lVert}{\rVert}
\newcommand{\R}{\mathbb{R}}
\newcommand{\N}{\mathbb{N}}

\newcommand{\C}{\mathbb{C}}

\newcommand{\eps}{\varepsilon}

\newcommand{\cJ}{\mathcal{J}}
\newcommand{\cL}{\mathcal{L}}
\newcommand{\dd}{\,\mathrm{d}}
\author[D. Sinambela]{Daniel Sinambela}
\address{Department of Mathematics, University of Utah, Asia Campus}
\email{daniel.sinambela@utah.edu}

\begin{document}
\title[Small-amplitude solitary waves on the $f$-plane]{Small-Amplitude Solitary Waves for $f$-Plane Capillary-Gravity Flows with Arbitrary Vorticity}
 \begin{abstract}

We study small-amplitude solitary waves for two-dimensional capillary--gravity flows with arbitrary vorticity on the equatorial $f$-plane. The steady free-boundary problem is formulated as a reversible Hamiltonian spatial-dynamics system in which rotation enters through the speed-dependent effective gravity $g_*=g-2\Omega c$. A center-manifold reduction reduces the local bifurcation problem to finite-dimensional Hamiltonian systems governed by the low-frequency spectrum of a Sturm--Liouville problem with an eigenvalue-dependent boundary condition. We identify the Hamiltonian $0^2$, real $1{:}1$, and Hamiltonian--Hopf resonance curves and obtain corresponding families of symmetric solitary waves under standard nondegeneracy assumptions. We also show that the weak-effective-gravity threshold $g_*=0$ is separated from the $0^2$ and local Hamiltonian--Hopf resonances for uniformly non-stagnant laminar flows, and can be approached only in a near-stagnation regime.

\end{abstract}

\maketitle
\tableofcontents

\section{Introduction}

Geophysical fluid dynamics is concerned with large-scale fluid phenomena for which the Earth's
rotation, known as the Coriolis force, plays an essential role. Among the most important settings is the equatorial region,
where the Coriolis effect strongly influences wave propagation and where zonal waves---waves propagating east to west---are observed
to be trapped near the equator and to decay away from it; see, for example,
\cite{CushmanRoisinBeckers2011,FedorovBrown2009,Pedlosky1979,
Constantin2012ExactEquatoriallyTrapped}.
This physical picture motivates the mathematical study of localized traveling waves in
equatorial water-wave models.

From a modeling perspective, the equatorial setting is particularly appealing. Around the equator,
the dynamics in the fluid are predominantly zonal, and for waves whose meridional length scale is small
relative to the Earth's radius, the full geophysical equations may be well approximated by the
$f$-plane equations, in which the Coriolis parameter is treated as constant; see
\cite{CushmanRoisinBeckers2011,Pedlosky1979}. In the present paper, we work with the
two-dimensional $f$-plane approximation with both surface tension and a general vorticity
distribution. Our interest is in small-amplitude solitary waves: traveling waves which
are localized. 

There has been substantial recent progress in the mathematical study of geophysical and
equatorial waves. For periodic equatorial waves and related geophysical flows, we refer, for
example, to
\cite{Constantin2013EquatorialWindWaves,henry2014existence,Henry2013UnderlyingCurrent,
Constantin2013ThreeDimensional,ConstantinGermain2013,Matioc2012EdgeWaves,
Matioc2012Coriolis,Matioc2013Stratified}. On the other hand, for classical water waves with
vorticity there is by now a plethora of works devoted to existence and bifurcation theory, both for periodic
and solitary waves; see
\cite{ConstStrauss2004,ConstantinVarvaruca2011,ConstantinStrauss2011Discontinuous,
Henry2013FixedDepth,Henry2013LargeAmplitude,shatah2013travelling,BuffoniToland2003,
sinambela2022existence} and the references therein. Solitary waves in these settings also enjoy
rich qualitative properties, including symmetry and refined asymptotic structure; see, for
instance,
\cite{HenryMatioc2014Symmetry,sinambela2025asymptotics,MatiocMatioc2013Symmetry}.

Our approach is inspired by the Hamiltonian spatial-dynamics framework developed by \cite{Mielke1988} and applied in \cite{GrovesWahlen2007}. In the framework, the
horizontal spatial direction is treated as the evolution variable, and the
steady water-wave problem is recast as an infinite-dimensional reversible
Hamiltonian system. Localized waves are then obtained as homoclinic orbits
on finite-dimensional center manifolds. This point of view is particularly
well suited to the study of low-frequency Hamiltonian resonances, such as
the Hamiltonian \(0^2\), real \(1{:}1\), and Hamiltonian--Hopf resonances.

The main structural observation in the present paper is that the steady
\(f\)-plane problem admits the same Dubreil--Jacotin height formulation as
the classical rotational capillary--gravity problem, except that the usual
gravitational coefficient is replaced by the speed-dependent effective
gravity
\begin{equation}\label{eq:intro-effective-gravity}
    g_*:=g-2\Omega c,
\end{equation}
where \(\Omega\) is the Coriolis parameter and \(c\) is the wave speed. Thus,
rotation does not merely add a lower-order perturbation to the equations; it
changes the bifurcation parameter itself. In the classical theory, gravity is
a fixed external parameter. In the \(f\)-plane problem, by contrast, the
effective restoring force depends on the wave speed. This coupling between
rotation, propagation speed, and gravity is the source of the new phenomena
analyzed here.

After nondimensionalization, the effective gravity is represented by
\[
    \alpha=\frac{g_*d^3}{m^2}
    =
    \frac{(g-2\Omega c)d^3}{m^2},
\]
where \(d\) is the undisturbed depth and \(m\) is the relative mass flux. The
distinguished threshold
\[
    g_*=0
    \qquad\Longleftrightarrow\qquad
    \alpha=0
\]
has no analogue in the non-rotating capillary--gravity theory. We refer to
this as the weak-effective-gravity regime. From the viewpoint of the
linearized spatial-dynamics problem, this regime is delicate because varying
the wave speed changes the affine boundary condition in the associated
scalar Sturm--Liouville problem. A central question is therefore whether the
classical low-frequency resonance curves persist unchanged, or whether the
\(f\)-plane coupling produces new spectral behavior near \(\alpha=0\).

We answer this question by combining the Hamiltonian formulation, a
center-manifold reduction in the sense of \cite{Mielke1988}, and a detailed
analysis of the low-frequency spectrum. The linearized spatial-dynamics
operator is reduced to a scalar Sturm--Liouville problem with an
eigenvalue-dependent boundary condition. The intersections and tangencies of
the corresponding boundary spectral curve determine the Hamiltonian \(0^2\),
real \(1{:}1\), and Hamiltonian--Hopf resonance curves in the
\((\beta,\alpha)\)-plane, where \(\beta\) is the nondimensional surface
tension parameter.

The novelty of the present \(f\)-plane problem appears in three related ways. First,
the effective gravity \(g_*=g-2\Omega c\) makes the gravity parameter
speed-dependent, so that perturbing the wave speed moves the problem through
the low-frequency spectral diagram. Second, the weak-effective-gravity
threshold \(\alpha=0\) is separated from the Hamiltonian \(0^2\) resonance
and from the local Hamiltonian--Hopf branch for laminar flows uniformly away
from stagnation. Thus these classical resonance mechanisms cannot approach
the weak-effective-gravity regime unless the underlying laminar flow becomes
near-stagnant. Third, the Hamiltonian \(0^2\) reduction yields a combined
Coriolis--capillarity criterion for the polarity of the bifurcating solitary
waves: depending on the side of the capillarity threshold \(\beta_*\) and
the Coriolis threshold \(\Omega_c\), the small-amplitude branch consists of
waves of elevation or waves of depression.

In this sense, the present work extends the classical spatial-dynamics
theory of rotational capillary--gravity solitary waves to an equatorial
\(f\)-plane setting and identifies how rotation changes the low-frequency
Hamiltonian resonance structure. The solitary waves constructed here
bifurcate from laminar flows whose velocity field is horizontal but
depth-dependent, and the vorticity distribution is allowed to be arbitrary
within the regularity class specified below.

The paper is organized as follows. In the next subsection we state the main
results. In Section~\ref{sec:governing}, we derive the steady \(f\)-plane
capillary--gravity equations, the stream-function formulation, the
Dubreil--Jacotin height equation, and the associated Hamiltonian structure.
Section~\ref{sec:CMR} verifies the center-manifold hypotheses and reduces
the spectral problem to a scalar Sturm--Liouville equation with an
eigenvalue-dependent boundary condition. Section~\ref{sec:Resonances}
analyzes the Hamiltonian \(0^2\), real \(1{:}1\), and Hamiltonian--Hopf
resonances and describes the corresponding reduced dynamics. Additionally, we also discuss about the criterion for solitary waves of depression and elevation. Finally,
Section~\ref{sec:weak-gravity} studies the weak-effective-gravity regime and
proves its separation from the relevant low-frequency resonance curves for
uniformly non-stagnant laminar flows.

\subsection{Statement of results}

We now state the main results. The first theorem gives the general
small-amplitude solitary-wave existence statement obtained from the
Hamiltonian spatial-dynamics formulation and center-manifold reduction.
\begin{theorem}[Existence of Small-Amplitude Solitary Waves]\label{thm:existence}
Let the vorticity distribution function satisfy $\gamma \in H^{1}(m,0)$, and assume the underlying laminar flow is strictly without stagnation, meaning $\inf_{D}(u-c) > 0$. Suppose that there exists a critical surface tension and wave speed $(\sigma_0, c_0)$ such that, at the effective gravity $g^0_{*} = g - 2\Omega c_0$, then the linearized spatial-dynamics operator exhibits either
\begin{enumerate}
    \item a Hamiltonian $0^2$ resonance,
    \item  a real $1:1$ resonance, or
    \item  a Hamiltonian--Hopf bifurcation.
\end{enumerate} Under standard non-degeneracy conditions, the fully nonlinear steady equatorial $f$-plane capillary-gravity problem admits a family of non-trivial, symmetric solitary waves bifurcating from this laminar flow for parameters near $(\sigma_0, c_0)$.
\end{theorem}
% \begin{theorem}[Small-amplitude solitary waves]
% Let $\gamma\in H^1(m,0)$ and suppose that the corresponding laminar flow
% satisfies the non-stagnation condition $u-c>0$
% throughout the fluid. Assume that, for some parameter value
% $(\beta_0,\alpha_0)$,
% the linearized spatial-dynamics operator possesses either

% \begin{enumerate}
% \item  Hamiltonian $0^2$ resonance,
% \item  Real $1{:}1$ resonance, or
% \item  Hamiltonian--Hopf,
% \end{enumerate}
% and that the corresponding non-degeneracy conditions hold. Then there exists a neighborhood of $(\beta_0,\alpha_0)$ in which the
% steady $f$-plane capillary--gravity problem admits a family of nontrivial
% symmetric solitary-wave solutions bifurcating from the laminar flow.
% \end{theorem}
\begin{remark}
The non-degeneracy conditions in theorem above stated in subsections~4.1--4.3: in the \(0^2\) case the cubic
coefficient has to be nonzero, in the real \(1:1\) case the reduced fourth-order
normal form admits a reversible homoclinic orbit, and in the Hamiltonian--Hopf
case the normal-form coefficients must satisfy \(c_1<0<c_3\).
\end{remark}

The Hamiltonian \(0^2\) normal form also determines the leading-order polarity of the
bifurcating solitary waves. The result is a combined Coriolis--capillarity
criterion: elevation and depression are selected by both the side of the
capillarity threshold \(\beta_*\) and the side of the Coriolis threshold
\(\Omega_c\).

\begin{corollary}[Solitary Waves of Depression and Elevation]\label{cor:threshold}
Let $g_{*}^{(0^2)}$ denote the critical effective gravity at which the
Hamiltonian $0^2$ resonance occurs, and let $c^{(0^2)}>0$ be the associated
critical wave speed. We define the critical Coriolis threshold by
\[
    \Omega_c:=\frac{g-g_{*}^{(0^2)}}{2c^{(0^2)}}.
\]
Let $\beta_*$ be the critical capillarity value corresponding to the
codimension-two point $(\beta_*,\alpha_*)$, and let $\beta_0\ne\beta_*$ be
the capillarity parameter along the Hamiltonian $0^2$ branch. Then the
polarity of the small-amplitude solitary waves bifurcating from the
Hamiltonian $0^2$ resonance is determined as follows:
\[
\begin{cases}
\text{solitary waves of elevation}, & \beta_0<\beta_*
\ \text{and}\ \Omega>\Omega_c,\\[0.6ex]
\text{solitary waves of depression}, & \beta_0>\beta_*
\ \text{and}\ \Omega<\Omega_c.
\end{cases}
\]
\end{corollary}
\begin{remark}
    The comparison with \(\Omega_c\) above is understood locally near the critical
wave speed \(c^{(0^2)}\), so that
\[
\alpha-\alpha_*
=
-\frac{2c^{(0^2)}d^3}{m^2}(\Omega-\Omega_c)
+\textup{higher-order terms}.
\]
\end{remark}
The next two theorems record our second contribution which is novel and tailored to the steady $f$-plane problem considered here. We show that in the weak-effective-gravity regime, the Hamiltonian $0^2$ resonance and Hamiltonian--Hopf phenomena disappear when working with flows with no stagnation points. The proofs of these theorems are recorded in Section~\ref{sec:weak-gravity}.
\begin{theorem}[Near-Stagnation Limit of the $0^2$ Resonance]\label{thm:nearstag}
Suppose that $c^{(0^2)}$ denotes the critical wave speed at which the linearized system exhibits a Hamiltonian $0^2$ resonance, and define the corresponding critical effective gravity as 
\[
g_{*}^{(0^2)} = g - 2\Omega c^{(0^2)}.
\] 
For any family of laminar flows satisfying \(a(\zeta)\ge a_{\min}>0\) uniformly,
the corresponding \(0^2\) critical effective gravity is bounded away from zero.
Consequently, \(g_*^{(0^2)}\to0\) can occur only along a family for which
\(\min_{\zeta\in[0,1]}a(\zeta)\to0\), equivalently the laminar flow approaches
stagnation.
\end{theorem}

% \begin{theorem}[Near-stagnation characterization]
% Let $g_*^{(0^2)}$ denote the effective gravity value corresponding to the $0^2$ resonance.
% Then
% \[
% g_*^{(0^2)}\to0
% \text{ if and only if }
% \inf_D (u-c)\to0.
% \]
% Hence, the weak-effective-gravity threshold can be approached by the $0^2$ resonance only through a near-stagnation regime.
% \end{theorem}
\begin{theorem}[Persistence and Separation from Weak-Effective Gravity]\label{thm:HH}
Let $H$ be an underlying laminar flow satisfying the strict non-stagnation condition $u-c>0$ uniformly throughout the fluid domain. Let $\sigma$ denote the surface tension coefficient and let $g_\ast = g - 2\Omega c$ denote the speed-dependent effective gravity. Assume that the natural nondegeneracy condition $\mathfrak{B}''(0)<0$ holds at the zero spectrum, where
\[
\mathfrak{B}(\nu)=\frac{a(1)^3M_\zeta(1;\nu)}{M(1;\nu)},
\]
where $M$ satisfies \eqref{eq:I-def-HH-proof}. Then the following properties hold:
\begin{enumerate}
    \item Smooth Persistence: The Hamiltonian–Hopf resonance curve persists smoothly as the imaginary spatial eigenvalue parameter $\mathfrak{q} \to 0$, terminating precisely at the codimension-two $0^4$ resonance point characterized by a critical surface tension $\sigma$ and a critical effective gravity $g^H_\ast$.
    \item Separation from Weak-Effective Gravity: The critical effective gravity at this codimension-two anchor point is strictly positive ($g^H_\ast > 0$). Consequently, the local Hamiltonian–Hopf curve is bounded away from, and cannot enter, the weak-effective-gravity regime where $g_\ast = g - 2\Omega c \approx 0$.
\end{enumerate}
\end{theorem}
The key new feature relative to the classical rotational capillary--gravity
problem is that the bifurcation parameter \(\alpha\) is not simply a fixed
gravity parameter, but depends on the wave speed through
\(g_*=g-2\Omega c\). This creates the weak-effective-gravity threshold
\(\alpha=0\), and the results above show that the classical low-frequency
resonances remain separated from this threshold unless the laminar flow
approaches stagnation.

%%%%%%%%%%%%%%%%%%

\section{Hamiltonian system formulation}\label{sec:governing}

We let $z=\eta(t,x,y)$ be the surface of the ocean and set $z=0$ as the mean surface level for the flow and $z=-d$ denoting the lower boundary of the layer to which equatorial undercurrent is confined. The $f$-plane governing equations are given as follows
\begin{equation}\label{time dependent fplane-euler}
\begin{cases}
(u_t+uu_x+vu_y+wu_z)+2\Omega w=-\dfrac{P_x}{\rho},\\[0.8ex]
v_t+uv_x+vv_y+wv_z=-\dfrac{P_y}{\rho},\\[0.8ex]
w_t+uw_x+vw_y+ww_z-2\Omega u=-\dfrac{P_z}{\rho}-g,
\end{cases}
\end{equation}
where $\Omega$ is the constant rotational speed of the earth around the polar axis towards east, $\rho$ is the density constant of the water, and $g$ is the gravitational acceleration constant. 
Under the assumption of constant density , the continuity equation reads:
\[
u_x+v_y+w_z=0.
\]
Additionally, we also have the following kinematic boundary condition on the free surface
\[
w=\eta_t+u\eta_x + v\eta_y \qquad \text{on } z=\eta(t,x,y).
\]
We assume that  the ocean floor is impermeable which translates to
\[
w=0, \qquad  \text{on } z=-d.
\]

In the present work, we are particularly interested in the two-dimensional waves moving the zonal direction along the equator line, independent of the $y$. This condition allows us to say that $v \equiv 0$ throughout the fluid. Following this assumption, therefore, the vertical vorticity is given by the scalar 
\[
\gamma=u_z-w_x.
\]

The fluid occupies the domain
\[
D:=\{(x,z)\in\R^2:-d<z<\eta(x)\}.
\]
In the moving frame of reference with speed $c>0$ combined with the vanishing assumption of the velocity component $v$, the bulk equations now become the steady $f$-plane equations
\begin{equation}\label{eq:fplane-euler}
\begin{cases}
(u-c)u_x+w u_z+2\Omega w=-\dfrac{P_x}{\rho},\\ \\[0.8ex]
(u-c)w_x+w w_z-2\Omega u=-\dfrac{P_z}{\rho}-g,
\end{cases}
\end{equation}
with bed condition
\begin{equation}\label{eq:bed-bc}
    w=0 \qquad \text{on } z=-d,
\end{equation}
and kinematic condition on the free surface
\begin{equation}\label{eq:kinematic}
    w=(u-c)\eta_x \qquad \text{on } z=\eta(x).
\end{equation}
The continuity equation  then becomes
\[
u_x+w_z=0.
\]
Furthermore, we impose the no-stagnation condition throughout the fluid which reads
\begin{equation}\label{eq:no-stag}
    u-c>0 \qquad \text{in } D.
\end{equation}Our convention here differs from the usual no-stagnation condition in the water-wave literature where commonly the inequality is given by $u-c<0$.

Additionally, we take into account the surface tension which plays a role as a restoring force on the free surface. It enters through the Laplace--Young law
\begin{equation}\label{eq:laplace-young}
    P=P_{\mathrm{atm}}-\sigma\left(\frac{\eta_x}{\sqrt{1+\eta_x^2}}\right)_x
    \qquad \text{on } z=\eta(x),
\end{equation}
where $\sigma>0$ is the coefficient of surface tension.

\subsection{Stream function formulation}
In the following subsection, we recast the $f$-plane problem in terms of the stream function $\psi$.
We define the stream function (up to a constant) $\psi$ by
\begin{equation}\label{eq:stream}
    \psi_z=u-c,
    \qquad
    \psi_x=-w.
\end{equation}
The no-stagnation condition \eqref{eq:no-stag} translates to $\psi_z>0$. We fix the additive constant so that
\begin{equation}\label{eq:stream-top}
    \psi=0 \qquad \text{on } z=\eta(x),
\end{equation}
and we let
\begin{equation}\label{eq:stream-bottom}
    \psi=m \qquad \text{on } z=-d,
\end{equation}
where $m<0$ is the relative mass flux given by
\[
\int_{-d}^{\eta(x)}(c-u(x,z))\;dz <0.
\]

Assuming that the vorticity is prescribed by a function of the streamlines,
\begin{equation}\label{eq:vorticity-law}
    u_z-w_x=\gamma(\psi),
\end{equation}
a standard Bernoulli calculation in the $f$-plane setting yields the hydraulic head
\[
E=\frac{(u-c)^2+w^2}{2}+(g-2\Omega c)z+\frac{P}{\rho}-2\Omega\psi+\int_0^{\psi}\gamma(s)\,\dd s, \qquad \text{in } D
\]
which is constant in the fluid. After absorbing the atmospheric pressure into the Bernoulli constant, the free-surface condition becomes
\begin{equation}\label{eq:bernoulli-stream}
    |\nabla\psi|^2+2g_* z-2\sigma\left(\frac{\eta_x}{\sqrt{1+\eta_x^2}}\right)_x=Q
    \qquad \text{on } z=\eta(x),
\end{equation}
where $g_*$ is defined by 
\begin{equation}\label{eq:effective-gravity}
    g_*:=g-2\Omega c.
\end{equation}
Therefore, the steady capillary $f$-plane problem takes the form
\begin{equation}\label{eq:stream-problem}
\begin{cases}
\Delta\psi=\gamma(\psi) & \text{in } -d<z<\eta(x),\\
\psi=0 & \text{on } z=\eta(x),\\
\psi=m & \text{on } z=-d,\\
|\nabla\psi|^2+2g_* z-2\sigma\left(\dfrac{\eta_x}{\sqrt{1+\eta_x^2}}\right)_x=Q & \text{on } z=\eta(x).
\end{cases}
\end{equation}

\subsection{Height formulation and variational structure and Euler--Lagrange}\label{sec:height}

We introduce the Dubreil--Jacotin hodograph transformation $\mathfrak{H}:D \rightarrow T$ defined by 
\[
\mathfrak{H}(x,z)=(x,\psi(x,z))=:(x,\zeta),
\] where $T:=\mathbb{R}\times [m,0]$. It is known  that $\mathfrak{H}$, is in fact a diffeomorphism see \cite{MartinMatioc2014}. Furthermore, in order to achieve the desired Hamiltonian formulation for our later analysis, we need to perform another transformation to \eqref{eq:stream-problem}. To that end, we define the height function $h:T\rightarrow \mathbb{R}$ as follows
\[
h(x,\zeta)=z+d.
\]

Via simple computations, we obtain standard identities involving the derivatives of $h$ and $\psi$, namely
\begin{equation}\label{eq:height-identities}
    h_x=-\frac{\psi_x}{\psi_z},
    \qquad
    h_\zeta=\frac{1}{\psi_z}>0.
\end{equation}
As a result in the $(x,\zeta)$ variables, the bulk equation in \eqref{eq:stream-problem} reads
\begin{equation}\label{eq:height-eq}
    (1+h_x^2)h_{\zeta\zeta}-2h_\zeta h_x h_{\zeta x}+h_\zeta^2h_{xx}+\gamma(\zeta)h_\zeta^3=0
    \qquad \text{in } T,
\end{equation}
while the condition on the bed becomes
\begin{equation}\label{eq:height-bed}
    h=0 \qquad \text{on } \zeta=m.
\end{equation}
At the free surface $\zeta=0$, one has $\eta(x)=h(x,0)-d$, and the Bernoulli equation in \eqref{eq:stream-problem} reads
\begin{equation}\label{eq:height-top}
    \frac{1+h_x^2}{2h_\zeta^2}+g_*(h-d)-\sigma\left(\frac{h_x}{\sqrt{1+h_x^2}}\right)_x=\frac{Q}{2}
    \qquad \text{on } \zeta=0.
\end{equation}

We define the function $\Gamma$, associated with the general vorticity distribution $\gamma$, by
\begin{equation}\label{eq:Gamma-def}
    \Gamma(\zeta):=\int_m^\zeta\gamma(s)\,\dd s.
\end{equation}
Then \eqref{eq:height-eq}, \eqref{eq:height-bed}, and \eqref{eq:height-top} can be viewed formally as the Euler--Lagrange equations of the action
\begin{equation}\label{eq:action}
    \cJ[h]=\int_{\R}\left[\int_m^0\left(-\frac{1+h_x^2}{2h_\zeta}+\Big(g_*(h-d)-\frac{Q}{2}-\Gamma(\zeta)\Big)h_\zeta\right)\dd \zeta+\sigma\sqrt{1+\eta_x^2}\right]\dd x,
\end{equation}
where $\eta(x)=h(x,0)-d$.

Following the capillary-gravity Hamiltonian construction in \cite{GrovesWahlen2007}, we introduce the Legendre variables
\begin{equation}\label{eq:legendre}
    \omega:=\frac{\delta\cJ}{\delta\eta_x}=\sigma\frac{\eta_x}{\sqrt{1+\eta_x^2}},
    \qquad
    w:=\frac{\delta\cJ}{\delta h_x}=-\frac{h_x}{h_\zeta},
\end{equation}
which yield
\begin{equation}\label{eq:legendre-inverse}
    \eta_x=\frac{\omega}{\sqrt{\sigma^2-\omega^2}},
    \qquad
    h_x=-h_\zeta w.
\end{equation}
It is not hard to see that the Legendre transform of \eqref{eq:action} gives rise to the Hamiltonian
\begin{equation}\label{eq:hamiltonian}
\mathcal{H}(\eta,\omega,h,w)=\int_m^0\left[\frac12\left(\frac{1}{h_\zeta}-h_\zeta w^2\right)+\Gamma(\zeta)h_\zeta\right]\dd \zeta-\frac{g_*}{2}(\eta)^2+\frac{Q}{2}(\eta+d)-\sqrt{\sigma^2-\omega^2},
\end{equation}
up to an additive constant.

We now formally define the underlying functional spaces of the Hamiltonian variables above. Let
\[
X:=\{(\eta,\omega,h,w)\in\R\times\R\times H^1(m,0)\times L^2(m,0): h(m)=0,\ h(0)=\eta+d\},
\]
\[
Y:=\{(\eta,\omega,h,w)\in\R\times\R\times H^2(m,0)\times H^1(m,0): h(m)=0,\ h(0)=\eta+d\},
\]
and we equip $X$ with the canonical symplectic form
\begin{equation}\label{eq:symplectic}
\bm{\mathfrak{A}}\big((\eta_1,\omega_1,h_1,w_1),(\eta_2,\omega_2,h_2,w_2)\big)=\int_m^0(w_2h_1-w_1h_2)\dd \zeta+\omega_2\eta_1-\omega_1\eta_2.
\end{equation}
We also assume that the vorticity function $\gamma \in L^2 (m,0)$, so that $\Gamma \in H^1(m,0)$. Observe that the space $Y$ is slightly more regular than $X$. Ultimately, the Hamiltonian vector field belongs to subspace of $Y$. That leads us to define the set $\mathcal{O} \subset Y$ by
\[
\mathcal{O}:=\{(\eta,\omega,h,w)\in Y: |\omega|< \sigma, |\eta|\leq \eta_0, h_\zeta>0, \text{ for fixed } \zeta\in (m,0) \text{ and } 0<\eta_0\ll1\}.
\]
The vector field obtained from Hamilton equations below belongs to the space $\mathcal{O}$. 

Via direct calculations, the Hamilton equations associated with the Hamiltonian \eqref{eq:hamiltonian} are given by
\begin{equation}\label{eq:hamilton-eqs}
\begin{cases}
\eta_x=\dfrac{\omega}{\sqrt{\sigma^2-\omega^2}},\\[0.8ex]
\omega_x=\dfrac12\left(w(0)^2+\dfrac{1}{h_\zeta(0)^2}\right)+g_*(\eta)-\dfrac{Q}{2},\\[0.8ex]
h_x=-h_\zeta w,\\[0.8ex]
w_x=-\dfrac12\left(w^2+\dfrac{1}{h_\zeta^2}\right)_\zeta+\gamma(\zeta),
\end{cases}
\end{equation}
with boundary conditions
\begin{equation}\label{eq:ham-domain}
    w(m)=0,
    \qquad
    h_\zeta(0)w(0)=-\frac{\omega}{\sqrt{\sigma^2-\omega^2}}.
\end{equation}
Notice that the system is reversible under the involution
\begin{equation}\label{eq:reverser}
    S(\eta,\omega,h,w)=(\eta,-\omega,h,-w).
\end{equation}
Here, reversibility means that the equations do not distinguish between moving to the right in $x$ and moving to the left in $x$, provided the momenta are sign-reversed.
Concretely, if $(\eta(x), \omega(x),h(x,\zeta), w(x,\zeta))$ solves the Hamiltonian system, then so does $(\eta(-x), -\omega(-x),h(-x,\zeta),\\ -w(-x,\zeta))$. Hence, we conclude that the steady $f$-plane capillary-gravity problem \eqref{eq:stream-problem} is formally equivalent to the reversible Hamiltonian system \eqref{eq:hamilton-eqs}--\eqref{eq:ham-domain} on $(X,\bm{\mathfrak{A}})$ with Hamiltonian \eqref{eq:hamiltonian}.

% \begin{proposition}\label{prop:ham-formulation}

% \end{proposition}

% \begin{proof}
% The derivation is exactly the Legendre-transform calculation recorded above. The bulk terms come from the height equation, while the capillary boundary term furnishes the momentum variable $\omega$ and prevents degeneracy of the surface Legendre transform.
% \end{proof}

%\section{Laminar flows and solitary-wave variables}\label{sec:laminar}

Since the bifurcation studied later is started at a laminar flow, it is therefore worth defining what this trivial flow is. The laminar flows are $x$-independent solutions of the height equation. They take the form
\begin{equation}\label{eq:laminar}
    H(\zeta;\lambda)=\int_m^\zeta\frac{\dd s}{\sqrt{\lambda+2\Gamma(s)}},
\end{equation}
where $\lambda$ is a constant.
For later use, we define 
\begin{equation}\label{eq:a-def}
    a(\zeta):=H_\zeta^{-1}(\zeta)=\sqrt{\lambda+2\Gamma(\zeta)}.
\end{equation}
The following lemma states that for any given vorticity distribution $\gamma$, there exists a unique constant $\lambda_\ast$ which gives rise to a laminar flow.
\begin{lemma}\label{laminarsol}
    Assume that the vorticity function $\gamma$ is such that $\Gamma \in H^1(m,0)$. Suppose that 
    \[
    \int_{m}^{0}\frac{ds}{\sqrt{2\Gamma(s)-2\Gamma_{\textup{min}}}}>d,
    \] where $\Gamma_{\textup{min}}:= \min_{s\in[m,0]} \Gamma(s)$ and $\Gamma_{\textup{max}}:= \max_{s\in[m,0]} \Gamma(s)$.
    Then there exists a unique $\lambda_\ast>-2\Gamma_{\textup{min}}$ for which the Hamilton equations \eqref{eq:hamilton-eqs} admits the laminar solution
    \[
    (\eta,\omega,h,w)=(0,0,H(\zeta;\lambda_\ast),0),
    \]
     $H(m;\lambda_\ast)=0, H(0;\lambda_\ast)=d$ and $Q=\lambda_\ast+2\Gamma(0)$.
\end{lemma}
\begin{proof}
%    \begin{lemma}
% Assume that the vorticity function $\gamma$ is such that
% $\Gamma\in H^1(m,0)$, where
% \[
%     \Gamma(\zeta):=\int_m^\zeta \gamma(s)\,ds.
% \]
% Let
% \[
%     \Gamma_{\min}:=\min_{s\in[m,0]}\Gamma(s),
%     \qquad
%     \Gamma_{\max}:=\max_{s\in[m,0]}\Gamma(s).
% \]
% Suppose that
% \[
%     \int_m^0
%     \frac{ds}{\sqrt{2\Gamma(s)-2\Gamma_{\min}}}
%     > d,
% \]
% where the integral is understood as an improper integral if the
% denominator vanishes. Then there exists a unique
% $\lambda_* > -2\Gamma_{\min}$ such that
% \[
%     H(\zeta;\lambda_*)
%     :=
%     \int_m^\zeta
%     \frac{ds}{\sqrt{\lambda_*+2\Gamma(s)}}
% \]
% satisfies
% \[
%     H(m;\lambda_*)=0,
%     \qquad
%     H(0;\lambda_*)=d.
% \]
% Moreover, the Hamilton equations admit the laminar solution
% \[
%     (\eta,\omega,h,w)
%     =
%     \bigl(0,0,H(\zeta;\lambda_*),0\bigr),
% \]
% provided the Bernoulli constant is chosen as
% \[
%     Q=\lambda_*+2\Gamma(0).
% \]
% \end{lemma}

Let $\lambda>-2\Gamma_{\min}$ and by definition
\[
   H(0;\lambda)
    =
    \int_m^0
    \frac{ds}{\sqrt{\lambda+2\Gamma(s)}}.
\]
The condition $\lambda>-2\Gamma_{\min}$ guarantees that $ \lambda+2\Gamma(s)>0 \text{ for all } s\in[m,0].$ Hence, $H(0;\lambda)$ is finite and positive. Moreover, if
$\lambda_1<\lambda_2$, then it is easy to see that
\[
    \frac{1}{\sqrt{\lambda_1+2\Gamma(s)}}
    >
    \frac{1}{\sqrt{\lambda_2+2\Gamma(s)}}
    \qquad \text{for all } s\in[m,0].
\]
Therefore, $H(0;\lambda)$ is strictly decreasing as a function of $\lambda$  on
$(-2\Gamma_{\min},\infty)$.

 By monotone convergence,
\[
    \lim_{\lambda\downarrow -2\Gamma_{\min}}
    H(0;\lambda)
    =
    \int_m^0
    \frac{ds}{\sqrt{2\Gamma(s)-2\Gamma_{\min}}},
\]
where the integral may be infinite. By using the fact that $H(0;\lambda)$ is decreasing in $\lambda$, this limit is strictly
larger than $d$. On the other hand,
\[
    0\leq H(0;\lambda)
    \leq
    \frac{|m|}{\sqrt{\lambda+2\Gamma_{\min}}},
\]
which implies
\[
    \lim_{\lambda\to\infty}H(0;\lambda)=0.
\]
Hence, the intermediate
value theorem gives a unique
\[
    \lambda_* > -2\Gamma_{\min}
\]
such that
\[
    H(0;\lambda_*)=
    \int_m^0
    \frac{ds}{\sqrt{\lambda_*+2\Gamma(s)}}
    =d,
\]
while clearly $H(m;\lambda_*)=0$.

It remains to verify that this choice gives a laminar solution of the
Hamilton equations. Let
\[
    h(\zeta)=H(\zeta;\lambda_*), \qquad \eta=0,\qquad \omega=0,\qquad w=0.
\]
Then,
\[
    h_\zeta(\zeta)=\frac{1}{\sqrt{\lambda_*+2\Gamma(\zeta)}}>0,
\]
so the no-stagnation condition is satisfied. Since the solution is
$x$-independent and $w=0$, the equations
\[
    \eta_x=\frac{\omega}{\sqrt{\sigma^2-\omega^2}},
    \qquad
    h_x=-h_\zeta w
\]
are automatically satisfied. By simple algebra calculations, the first, third, fourth equations in \eqref{eq:hamilton-eqs} and boundary conditions \eqref{eq:ham-domain} are clearly satisfied. It remains to check that the laminar solution satisfies the second Hamilton equation. The first and third equations in \eqref{eq:hamilton-eqs} together with the boundary conditions \eqref{eq:ham-domain} clearly hold for $\eta=0$, $w=0$ and $\omega=0$. Moreover,
substituting $\eta=0$, $w=0$, and
\[
    \frac{1}{H_\zeta(0)^2}
    =
    \lambda_*+2\Gamma(0),
\]
we obtain
\[
    \omega_x
    =
    \frac12\bigl(\lambda_*+2\Gamma(0)\bigr)
    -\frac{Q}{2}.
\]
Thus $\omega_x=0$ precisely when
\[
    Q=\lambda_*+2\Gamma(0).
\]
It remains to show that the fourth equation \eqref{eq:hamilton-eqs} also holds. Note, using the fact that $\Gamma'(\zeta)=\gamma(\zeta),$ $H_\zeta(\zeta)=\frac{1}{\sqrt{\lambda_\ast+2\Gamma(\zeta)}}, $ and $H_{\zeta\zeta}(\zeta)=\frac{-\gamma(\zeta)}{(\lambda_\ast+2\Gamma(\zeta))^{3/2}},$ we obtain 
\begin{equation}
    \begin{aligned}
        -\dfrac12\left(\dfrac{1}{H_\zeta^2}\right)_\zeta+\gamma(\zeta)&=-\dfrac12\left(\frac{-2 H_{\zeta\zeta}}{H_\zeta^3}\right)+\gamma(\zeta)\\&
        =0.
    \end{aligned}
\end{equation}
Hence, the fourth equation is satisfied.
 Therefore,
\[
    (\eta,\omega,h,w)
    =
    \bigl(0,0,H(\zeta;\lambda_*),0\bigr)
\]
is a laminar solution. The uniqueness of $\lambda_*$ follows from the
strict monotonicity of $H(0;\lambda)$.
\end{proof}

\subsection{Nondimensionalization}\label{sec:nondim}

To prepare for the spectral analysis of the linearized operator as well as the center manifold analysis, we need to 
rewrite the problem in dimensionless variables. The natural vertical length scale is the
undisturbed depth $d$, while the natural stream-function scale is the relative mass flux $m$.
Accordingly, we introduce
\begin{equation}\label{eq:nondim-vars}
x=d\,\tilde x,
\qquad
z=d\,\tilde z,
\qquad
\eta=d\,\tilde \eta,
\qquad
h=d\,\tilde h,
\qquad
\zeta=m(1-\tilde \zeta).
\end{equation}
Notice that \(\tilde \zeta=0\) corresponds to the bed and \(\tilde \zeta=1\) corresponds to the free surface.
For the momentum variable and the physical coefficients, we set
\begin{equation}\label{eq:nondim-params}
\omega=\frac{m^2}{d}\,\tilde \omega,
\qquad
\sigma=\frac{m^2}{d}\,\tilde \beta,
\qquad
g_{*}=\frac{m^2}{d^3}\,\tilde \alpha,
\qquad
Q=\frac{m^2}{d^2}\,\tilde \mu.
\end{equation}

In the remainder of the article, we will use the nondimensionalized variables throughout. However, we will abuse the notation by dropping ``tilde" when writing dimensionless variables or parameters. To that end, we introduce the main dimensionless parameters 
\begin{equation}\label{eq:alpha-beta-mu-def}
\alpha:=\frac{g_*d^3}{m^2},
\qquad
\beta:=\frac{\sigma d}{m^2},
\qquad
\mu:=\frac{Qd^2}{m^2}.
\end{equation}
 Additionally, the rescaled vorticity and its integral are denoted again by \(\gamma\) and \(\Gamma\), respectively.
 
Therefore, in the dimensionless setting, the Hamiltonian system \eqref{eq:hamilton-eqs} becomes
\begin{equation}\label{eq:hamilton-eqs-nondim}
\begin{cases}
\eta_x=\dfrac{\omega}{\sqrt{\beta^2-\omega^2}},\\[0.8ex]
\omega_x=\dfrac12\left(w(1)^2+\dfrac{1}{h_\zeta(1)^2}\right)+\alpha \eta-\dfrac{\mu}{2},\\[0.8ex]
h_x=-h_\zeta w,\\[0.8ex]
w_x=-\dfrac12\left(w^2+\dfrac{1}{h_\zeta^2}\right)_\zeta+\gamma(\zeta),
\end{cases}
\end{equation}
with boundary constraints
\begin{equation}\label{eq:ham-domain-nondim}
w(0)=0,
\qquad
h_\zeta(1)w(1)=-\frac{\omega}{\sqrt{\beta^2-\omega^2}}.
\end{equation}

In this paper, we are interested in solitary wave solutions that are obtained from perturbing the laminar flows.  Mathematically, we are looking for solutions of \eqref{eq:hamilton-eqs-nondim} and \eqref{eq:ham-domain-nondim} in the following  form
\begin{equation}\label{eq:shift}
    \eta(x)=\rho(x),
    \qquad
    h(x,\zeta)=H(\zeta)+\phi(x,\zeta).
\end{equation}
This forces $\phi$ to satisfy the conditions
\[
\phi(0)=0,
\qquad
\phi(1)=\rho,
\qquad
h_\zeta=a^{-1}+\phi_\zeta.
\]
In the shifted variables, after substituting \eqref{eq:shift} and $h_\zeta=a^{-1}+\phi_\zeta$, the entire system \eqref{eq:hamilton-eqs-nondim} and \eqref{eq:ham-domain-nondim} becomes
\begin{equation}\label{eq:shifted-system}
\begin{cases}
\rho_x=\dfrac{\omega}{\sqrt{\beta^2-\omega^2}},\\[0.8ex]
\omega_x=\dfrac12\left(w(1)^2+\dfrac{a(1)^2}{(1+a(1)\phi_\zeta(1))^2}\right)+\alpha\rho-\dfrac{\mu}{2},\\[0.8ex]
\phi_x=-(a^{-1}+\phi_\zeta)w,\\[0.8ex]
w_x=-\dfrac12\left(w^2+\dfrac{a(\zeta)^2}{(1+a(\zeta)\phi_\zeta)^2}\right)_\zeta+\gamma(\zeta),
\end{cases}
\end{equation}
with boundary condition
\begin{equation}\label{eq:shifted-bc}
    w(0)=0,
    \qquad
    (a(1)^{-1}+\phi_\zeta(1))w(1)=-\frac{\omega}{\sqrt{\beta^2-\omega^2}}.
\end{equation} Notice that the origin $(\rho,\omega,\phi,w)=(0,0,0,0)$ is the equilibrium point of the Hamilton equations above. 
%\section{Reference point}

Recall that from equation \eqref{eq:alpha-beta-mu-def}, the variable
\(\alpha\) plays the role of the  effective-gravity
parameter and \(\beta\) plays the role of the capillarity parameter. We now choose a fixed reference point $(\beta_0,\alpha_0)$
in the \((\beta,\alpha)\)-parameter plane and write the nearby parameter
values as
\begin{equation}\label{eq: perturbation param}
     (\beta,\alpha)
    =
    (\beta_0+\varepsilon_1,\alpha_0+\varepsilon_2),\quad \text{where } \varepsilon=(\varepsilon_1,\varepsilon_2).
\end{equation}
In practice, as seen later, the reference point \((\beta_0,\alpha_0)\) will be chosen so
that the linearized spatial-dynamics problem has the desired critical spectral
configuration. To put things into perspective, if \(g\), \(\Omega\), \(d\), and \(m\) are fixed (but not the wave speed $c$), then perturbing
\(\alpha\) above is equivalent to perturbing \(c\). Indeed,
for some wave speed $c_0$,
\[
    \alpha_0
    =
    \frac{(g-2\Omega c_0)d^3}{m^2},
\]
then the perturbation $\varepsilon_2$ can be thought of as a perturbation away from the reference wave speed
\[
    \varepsilon_2
    =
    \alpha-\alpha_0
    =
    -\frac{2\Omega d^3}{m^2}(c-c_0).
\]

By Lemma~\ref{laminarsol}, suppose that the admissible constant $\mu=\mu_{\ast}$ is chosen that gives rise to a laminar flow. Since the Bernoulli constant is fixed here, it can no longer be treated as an independent bifurcation parameter as is commonly done when doing bifurcation approach see \cite{ConstStrauss2004}. Writing $\beta$ and $\alpha$ using $\beta_0$ and $\alpha_0$, and fixing the Bernoulli constant $\mu_\ast$ yields the following Hamilton equations
\begin{equation}\label{perturbed HE}
\begin{cases}
\rho_x=\frac{\omega}{\sqrt{(\beta_0+\varepsilon_1)^2-\omega^2}},\\[0.8ex]
\omega_x=\frac12\left[w(1)^2+\frac{a(1)^2}{\left(1+a(1)\phi_\zeta(1)\right)^2}\right]+(\alpha_0+\varepsilon_2)\rho-\frac{\mu_\ast}{2},\\[0.8ex]
 \phi_x=-\left(a(\zeta)^{-1}+\phi_\zeta\right)w,\\[0.8ex]
 w_x=-\frac12\left[w^2+\frac{a(\zeta)^2}{\left(1+a(\zeta)\phi_\zeta\right)^2}\right]_\zeta +\gamma(\zeta),
\end{cases}
\end{equation}

with boundary condition
\begin{equation}\label{eq:bc-reference point}
    w(0)=0,
    \qquad
    (a(1)^{-1}+\phi_\zeta(1))w(1)=-\frac{\omega}{\sqrt{(\beta_0+\varepsilon_1)^2-\omega^2}}.
\end{equation}
We note that the perturbations $(\varepsilon_1,\varepsilon_2)$ are small and chosen in the neighborhoods so that
\[
    |\varepsilon_1|<\frac{\beta_0}{4},\qquad   |\omega|<\frac{\beta_0}{2}<\beta_0+\varepsilon_1.
\]
The last inequality above ensures that the quantity $\sqrt{(\beta_0+\varepsilon_1)^2-\omega^2}$ is real and bounded away from zero.
We also impose the condition
\[
    a(\zeta)^{-1}+\phi_\zeta(\zeta)>0,
    \qquad 0\leq \zeta\leq 1,
\]
so that the height transform remains invertible. Equivalently, this inequality is simply 
 the no-stagnation condition $h_\zeta(\zeta)>0$, but expressed in the shifted variables.
\section{Center-manifold reduction}\label{sec:CMR}
\subsection{Application of the center-manifold reduction theorem}
Here, we outline the center manifold reduction theorem of \cite{Mielke1988}. We list the necessary hypotheses that must be satisfied before applying the center-manifold theorem to our problem. 
\begin{theorem}\label{CMT}
    Consider the differential equation
    \begin{equation}\label{DE problem}
         u_x=\mathcal{L}u+\mathcal{N}(u;\varepsilon),
    \end{equation}
    which represents Hamilton equations for the reversible Hamiltonian system $(\mathcal{X}, \Omega^\varepsilon, H^\varepsilon)$, where $\Omega^\varepsilon$ is the symplectic form. Here, the solution $u$ belongs to a Hilbert space $\mathcal{X},\lambda \in \mathbb{R}^l$ is a parameter and $\mathcal{L}:D(\mathcal{L})\subset \mathcal{X}\to \mathcal{X}$ is a densely defined, closed linear operator. Suppose that $0$ is an equilibrium point of \eqref{DE problem} when $\varepsilon=0$ and that 

\begin{enumerate}
    \item[]\label{H1} $(A1)$ The part of the spectrum $\sigma(\mathcal{L})$ of $\mathcal{L}$ which lies on the imaginary axis consists of a finite number of eigenvalues of finite multiplicity and is separated from the rest of $\sigma(\mathcal{L})$ in the sense of Kato, so that $\mathcal{X}$ admits decomposition $\mathcal{X}=\mathcal{X}_1 \oplus \mathcal{X}_2$, where $\mathcal{X}_1 =\mathcal{P}(\mathcal{X}),$ $\mathcal{X}_2=(\mathcal{I}-\mathcal{P})(\mathcal{X})$ and $\mathcal{P}$ is the spectral projection corresponding to the purely imaginary part of $\sigma(\mathcal{L})$.
    \\
    \item[]\label{H2} $(A2)$ The operator $\mathcal{L}_2=\mathcal{L}|_{\mathcal{X}_2}$ satisfies the resolvent estimate 
    \[
    \norm{(\mathcal{L}_2-i \xi \mathcal{I})^{-1}}_{\mathcal{X}_2 \to \mathcal{X}_2}\leq \frac{C}{1+|\xi|}, \qquad \xi \in \mathbb{R}
    \]
     for some constant $C$ independent of $\xi$.
     \\
     \item[]\label{H3} $(A3)$ There exists a natural number $k$ and neighborhoods $\Lambda \subset \mathbb{R}^l$ of $0$ and $U \subset D(\mathcal{L})$ of $0$ such that $\mathcal{N}$ is $(k+1)$ times  continuously differentiable on $U \times \Lambda$ and $\mathcal{N}(0,0)=0,\; d_1 \mathcal{N}[0,0]=0.$
\end{enumerate}
Under these hypotheses, there exist neighborhoods $\tilde{\Lambda}\subset \Lambda$ and $\tilde{U}_1 \subset U \cap \mathcal{X}_1, \; \tilde{U}_2 \subset U \cap \mathcal{X}_2 $ of $0$ and a reduction function $r:\tilde{U}_1 \times \tilde{\Lambda} \to \tilde{U}_2 $ with the following properties. The reduction function $r$ is $k$ times continuously differentiable on $\tilde{U}_1 \times \tilde{\Lambda}$, its derivatives are bounded and uniformly continuous on $\tilde{U}_1 \times \tilde{\Lambda}$ and $r(0,0)=0,\quad d_1r[0;0]=0$. The graph $\tilde{X}^\varepsilon=\{\tilde{u}_1+r(\tilde{u}_1;\varepsilon)\in \tilde{U}_1 \times \tilde{U}_2: \tilde{u}_1\in\tilde{U}_1\}$ is a Hamiltonian center manifold for \eqref{DE problem}, so that
\begin{itemize}
    \item $\tilde{X}^\varepsilon$ is a locally invariant manifold of \eqref{DE problem}: through every point in $\tilde{X}^\varepsilon$ there passes a unique solution of \eqref{DE problem} that remains on $\tilde{X}^\varepsilon$ as long as it remains in $\tilde{U}_1 \times \tilde{U}_2.$
    \\
    \item Every small bounded solution $u(x)$, $x\in \mathbb{R}$ of \eqref{DE problem} that satisfies $(\tilde{u}_1(x),\tilde{u}_2(x)) \in \tilde{U}_1 \times \tilde{U}_2$ lies completely in $\tilde{X}^\varepsilon$.\\
    \item Every solution $\tilde{u}_1:(x_1,x_2) \to \tilde{U}_1$ of the reduced equation 
    \begin{equation}\label{reduceeq}
        \tilde{u}_{1x}=\mathcal{L}\tilde{u}_1 +\mathcal{PN}(\tilde{u}_1+r(\tilde{u}_1;\varepsilon);\varepsilon)
    \end{equation}
    generates a solution
    \begin{equation}
        u(x)=\tilde{u}_1(x)+r(\tilde{u}_1(x);\varepsilon)
    \end{equation}
    of the full equation \eqref{DE problem}.
    \\
    \item $\tilde{X}^\varepsilon$ is a symplectic submanifold of $X$ and the flow determined by the Hamiltonian system $(\tilde{X}^\varepsilon, \tilde{\Omega}^\varepsilon, H^\varepsilon),$ where the tilde denotes restriction to $\tilde{X}^\varepsilon$, coincides with the flow on $\tilde{X}^\varepsilon$ determined by $(X,\Omega\tilde,H^\varepsilon)$. The reduced equation \eqref{reduceeq} is reversible and represents Hamilton equations for  $\tilde{X}^\varepsilon, \tilde{\Omega}^\varepsilon, H^\varepsilon).$
\end{itemize}
\end{theorem}
\subsection{ Change of variables} Prior to applying the center manifold theorem to our problem, we still need to perform one more change of variable to linearize the second (nonlinear) boundary condition in \eqref{eq:bc-reference point}. We let 
\begin{equation}\label{eq:change of coord bc}
    \tau:=\frac{-a(1)\omega}{\sqrt{(\beta_0+\varepsilon_1)^2-\omega^2}}, \qquad \mathfrak{w}(\zeta):=(1+a(\zeta)\phi_\zeta(\zeta))w(\zeta).
\end{equation}
As a result, we obtain the following proposition.
\begin{proposition}\label{prop:flattened-system}
Let
\[
G^\eps(\rho,\omega,\phi,w)=(\rho,\tau,\phi,\mathfrak{w})
\]
be the change of variables in \eqref{eq:change of coord bc},
then the system \eqref{perturbed HE}
in the variables $(\rho,\omega,\phi,w)$ is transformed into the system
\begin{equation}\label{eq:flattened-rho}
\rho_x=-a(1)^{-1}\tau,
\end{equation}
\begin{equation}\label{eq:flattened-tau}
\tau_x=
-\frac{(a(1)^2+\tau^2)^{3/2}}{a(1)^2(\beta_0+\eps_1)}
\left[
\frac12\,\frac{\mathfrak{w}(1)^2+a(1)^2}{(1+a(1)\phi_\zeta(1))^2}
+(\alpha_0+\eps_2)\rho-\frac{\mu_*}{2}
\right],
\end{equation}
\begin{equation}\label{eq:flattened-phi}
\phi_x=-a(\zeta)^{-1}\mathfrak{w},
\end{equation}
\begin{equation}\label{eq:flattened-z}
\mathfrak{w}_x=
-\frac{a(\zeta)\mathfrak{w}\,\bigl(a(\zeta)^{-1}\mathfrak{w}\bigr)_\zeta}{1+a(\zeta)\phi_\zeta}
-(1+a(\zeta)\phi_\zeta)\left[
\frac12\left(
\frac{\mathfrak{w}^2+a(\zeta)^2}{(1+a(\zeta)\phi_\zeta)^2}
\right)_\zeta-\gamma(\zeta)
\right],
\end{equation}
with new boundary conditions
\begin{equation}\label{eq:flattened-bc}
\mathfrak{w}(0)=0,
\qquad
\mathfrak{w}(1)=\tau,
\qquad
\phi(0)=0,
\qquad
\phi(1)=\rho.
\end{equation}

Moreover, these equations represent Hamilton equations for the transformed Hamiltonian
system
\[
(X,\Phi^\eps,K^\eps),
\]
where
\[
K^\eps:=\mathcal{H}^\eps\circ (G^\eps)^{-1},
\qquad
\Phi^\eps:=\bigl((G^\eps)^{-1}\bigr)^*\bm{\mathfrak{A}}.
\]
The reverser is given by
\[
S(\rho,\tau,\phi,\mathfrak{w})=(\rho,-\tau,\phi,-\mathfrak{w}).
\]
\end{proposition}

\begin{proof}
To obtain equations \eqref{eq:flattened-rho}-\eqref{eq:flattened-bc}, one just needs to perform standard but tedious algebra. To get \eqref{eq:flattened-rho} and \eqref{eq:flattened-phi} is straightforward. We are going derive instead \eqref{eq:flattened-tau}. Observe that 
\begin{equation}\label{tq}
\begin{aligned}
  \tau_x&=\frac{-a(1)\omega_x}{((\beta_0+\varepsilon_1)^2-\omega^2)
}+\frac{a(1)\omega^2\omega_x}{((\beta_0+\varepsilon_1)^2-\omega^2)^\frac{3}{2}}\\
&=\left(-a(1)\frac{(\beta_0+\varepsilon_1)^2}{((\beta_0+\varepsilon_1)^2-\omega^2)^\frac{3}{2}}\right) \omega_x.
\end{aligned}
\end{equation}
Now, we need to check that  the factor next to $\omega_x$ is equivalent to the factor outside the square brackets in \eqref{eq:flattened-tau}. It is not hard to see that
\[
((\beta_0+\varepsilon_1)^2-\omega^2)=\frac{a^2(1)(\beta_0+\varepsilon_1)^2}{a^2(1)+\tau^2}.
\]
Thus, we get
\[
\left(-a(1)\frac{(\beta_0+\varepsilon_1)^2}{((\beta_0+\varepsilon_1)^2-\omega^2)^\frac{3}{2}}\right)=-\frac{(a^2(1)+\tau^2)^{\frac{3}{2}}}{a^2(1)(\beta_0+\varepsilon_1)}.
\]
The term $\omega_x$ in \eqref{tq} is almost similar to the expression of $\omega_x$ in \eqref{perturbed HE}, except that the terms inside the square bracket in the equation of $\omega_x$ in \eqref{perturbed HE} is now equivalent to 
\[
\frac{\mathfrak{w}(1)^2+a(1)^2}{(1+a(1)\phi_\zeta(1))^2},
\]
using the second coordinate transformation in \eqref{eq:change of coord bc}. This yields \eqref{eq:flattened-tau}.  To get to $\mathfrak{w}_x$, one needs to perform standard algebraic manipulations. Hence, we omit its derivation here.

Since $G^\eps$ is a local diffeomorphism, the transformed equations are Hamilton equations
for the pullback symplectic form and pullback Hamiltonian,
\[
\Phi^\eps:=\bigl((G^\eps)^{-1}\bigr)^*\bm{\mathfrak{A}},
\qquad
K^\eps:=\mathcal{H}^\eps\circ (G^\eps)^{-1}.
\]
Explicitly, we obtain
\begin{equation}\label{eq:symplectic form flattened}
\begin{aligned}
&\Phi^\eps_{(\rho,\tau,\phi,\mathfrak{w})}
\bigl((\rho_1,\tau_1,\phi_1,\mathfrak{w}_1),(\rho_2,\tau_2,\phi_2,\mathfrak{w}_2)\bigr)\\
&\qquad =
\int_0^1
\left[
\frac{\mathfrak{w}_2\phi_1-\mathfrak{w}_1\phi_2}{1+a(\zeta)\phi_\zeta}
-
\frac{a(\zeta)\mathfrak{w}}{(1+a(\zeta)\phi_\zeta)^2}
\bigl(\phi_{2\zeta}\phi_1-\phi_{1\zeta}\phi_2\bigr)
\right]\dd \zeta\\
&\qquad\quad
-\frac{(\beta_0+\eps_1)a(1)^2}{\bigl(a(1)^2+\tau^2\bigr)^{3/2}}
\bigl(\tau_2\rho_1-\tau_1\rho_2\bigr),
\end{aligned}
\end{equation}
and 
\begin{equation}\label{eq:Keps-fplane}
\begin{aligned}
K^\eps(\rho,\tau,\phi,\mathfrak{w})
&=
\int_0^1
\left[
\frac{a(\zeta)-a(\zeta)^{-1}\mathfrak{w}(\zeta)^2}{2\bigl(1+a(\zeta)\phi_\zeta(\zeta)\bigr)}
+\Gamma(\zeta)\phi_\zeta(\zeta)
-\frac{a(\zeta)}{2}
\right]\dd \zeta
-\frac{\alpha_0+\eps_2}{2}\rho^2
+\frac{\mu_*}{2}\rho
\\& \qquad -\frac{a(1)(\beta_0+\eps_1)}{\sqrt{a(1)^2+\tau^2}}
+(\beta_0+\eps_1).
\end{aligned}
\end{equation}
The reversibility follows directly from the definitions of $\tau$ and $\mathfrak{w}$.
\end{proof}
Note, the  domain of the vector field from the above Hamiltonian system is given by 
\[
\tilde{Y}=\{(\rho,\tau,\phi,\mathfrak{w}) \in\R\times\R\times H^2(0,1)\times H^1(0,1): \phi(0)=0,\ \phi(1)=\rho,\ \mathfrak{w}(0)=0,\ \mathfrak{w}(1)=\tau\}.
\]
After linearizing the Hamilton equations \eqref{eq:flattened-rho}-\eqref{eq:flattened-z} around the equilibrium   $u:=(\rho,\tau,\phi,\mathfrak{w})=(0,0,0,0)$, we recast the system in the form of a differential equation motivated by \eqref{DE problem}
\[
u_x=\mathcal{L}u+\mathcal{N}^\varepsilon(u),
\]
where the linear operator $\mathcal{L}:\tilde{Y}\rightarrow \tilde{X}$ is given by 

\begin{equation}\label{eq:L-flattened}
\mathcal{L}
\begin{pmatrix}
\rho\\[0.3ex]
\tau\\[0.3ex]
\phi\\[0.3ex]
\mathfrak{w}
\end{pmatrix}
=
\begin{pmatrix}
-\,a(1)^{-1}\tau\\[0.8ex]
a(1)^4\beta_0^{-1}\phi_\zeta(1)-a(1)\alpha_0\beta_0^{-1}\rho\\[0.8ex]
-\,a(\zeta)^{-1}\mathfrak{w}(\zeta)\\[0.8ex]
\bigl(a(\zeta)^3\phi_\zeta\bigr)_\zeta
\end{pmatrix},
\end{equation}
with domain
\begin{equation}\label{eq:D-L-flattened}
D(\mathcal{L})=
\Bigl\{
(\rho,\tau,\phi,\mathfrak{w}):
\phi\in H^2(0,1),\ \mathfrak{w}\in H^1(0,1),\ \phi(0)=0,\ \phi(1)=\rho,\ \mathfrak{w}(0)=0,\ \mathfrak{w}(1)=\tau
\Bigr\},
\end{equation}
and
\[
\tilde{X}:=\{(\rho,\tau,\phi,\mathfrak{w}) \in\R\times\R\times H^1(0,1)\times L^2(0,1): \phi(0)=0,\ \phi(1)=\rho
\}.
\]
Meanwhile, the term $\mathcal{N}$ contains all nonlinear and parameter-dependent higher-order
terms and satisfies
\[
    \mathcal{N}^0(0)=0,
    \qquad
    D_u\mathcal{N}^0(0)=0.
\]

Now, we verify that the hypotheses $(A1), (A2)$, and $(A3)$ in the Center Manifold theorem for the operator $\mathcal{L}$ above are satisfied. The third Hypothesis  $(A3)$ is straightforward, simply a regularity check. One can use tools such as Sobolev embedding to get up to the $C^{k+1}$ regularity. One can even prove $C^\infty$ regularity. We skip the proof for $(A3)$ here. However, we will provide the proof showing that $(A1)$ and $(A2)$ are satisfied in the present work. 

Consider the Eigenvalue problem 
$\mathcal{L}u=\kappa u$, where $\mathcal{L}$ is given as in \eqref{eq:L-flattened} and $u=(\rho, \tau, \phi,\mathfrak{w})$. We convert the eigenvalue problem for $\mathcal{L}$ to a scalar problem. This is precisely the content of the next proposition.

\begin{proposition}\label{prop:eigenvalue-reduction}
The eigenvalue problem
\[
\mathcal{L}u=\kappa u,
\qquad
u=(\rho,\tau,\phi,\mathfrak{w}) \in D(\mathcal{L}),
\]
is equivalent to the scalar boundary-value problem
\begin{equation}\label{eq:phi-eigenvalue-problem}
-\,a(\zeta)^{-1}\bigl(a(\zeta)^3\phi_\zeta\bigr)_\zeta=\kappa^2\phi,
\qquad 0<\zeta<1,
\end{equation}
with boundary conditions
\begin{equation}\label{eq:phi-eigenvalue-bc}
\phi(0)=0,
\qquad
-\frac{a(1)^3}{\beta_0}\phi_\zeta(1)+\frac{\alpha_0}{\beta_0}\phi(1)=\kappa^2\phi(1).
\end{equation}
\end{proposition}

\begin{proof}
Suppose that $\mathcal{L}u=\kappa u$. Writing out the components explicitly gives
\begin{align}
-\,a(1)^{-1}\tau &= \kappa \rho,\label{eq:eig1}\\
a(1)^4\beta_0^{-1}\phi_\zeta(1)-a(1)\alpha_0\beta_0^{-1}\rho&=\kappa \tau,\label{eq:eig2}\\
-\,a(\zeta)^{-1}\mathfrak{w}&=\kappa \phi,\label{eq:eig3}\\
\bigl(a(\zeta)^3\phi_\zeta\bigr)_\zeta&=\kappa \mathfrak{w},\label{eq:eig4}
\end{align}
together with the boundary conditions
\[
\phi(0)=0,\qquad \mathfrak{w}(0)=0,\qquad \mathfrak{w}(1)=\tau.
\]

From \eqref{eq:eig3} we obtain $\mathfrak{w}=-\kappa a(\zeta)\phi.$
Substituting this into \eqref{eq:eig4} yields
\[
\bigl(a(\zeta)^3\phi_\zeta\bigr)_\zeta=-\kappa^2 a(\zeta)\phi,
\]
which is equivalent to \eqref{eq:phi-eigenvalue-problem}. Since $\mathfrak{w}(1)=\tau$, we also have $
\tau=-\kappa a(1)\phi(1).$
Substituting this into \eqref{eq:eig2} and using $\rho=\phi(1)$ gives
\[
a(1)^4\beta_0^{-1}\phi_\zeta(1)-a(1)\alpha_0\beta_0^{-1}\phi(1)
=
-\kappa^2 a(1)\phi(1),
\]
and division by $a(1)$ yields \eqref{eq:phi-eigenvalue-bc}. The converse follows by reversing the calculation.
\end{proof}
\begin{proposition}\label{prop:spectrum-fplane}
The spectrum of the linearized operator
\[
\mathcal L : D(\mathcal L)\subset X \to X
\]
consists of isolated eigenvalues, each of which is geometrically simple and has finite
algebraic multiplicity.
\end{proposition}

\begin{proof}
Recall from proposition~\ref{prop:eigenvalue-reduction}, the eigenvalue problem for \(\mathcal L\) is equivalent to the scalar boundary-value
problem
\begin{equation}\label{eq:scalar-eig-problem-fplane}
-\,a(\zeta)^{-1}\bigl(a(\zeta)^3\phi_\zeta\bigr)_\zeta=\kappa^2\phi,
\qquad 0<\zeta<1,
\end{equation}
with boundary conditions
\begin{equation}\label{eq:scalar-eig-bc-fplane}
\phi(0)=0,
\qquad
-\frac{a(1)^3}{\beta_0}\phi_\zeta(1)+\frac{\alpha_0}{\beta_0}\phi(1)=\kappa^2\phi(1).
\end{equation}

We now show that the spectrum of \(\mathcal L\) is discrete. For each fixed
\(\kappa\in\mathbb C\), let \(\phi(\,\cdot\,;\kappa)\) denote the unique solution of the
initial-value problem
\[
-\,a(\zeta)^{-1}\bigl(a(\zeta)^3\phi_\zeta\bigr)_\zeta=\kappa^2\phi,
\qquad
\phi(0;\kappa)=0,
\qquad
\phi_\zeta(0;\kappa)=1.
\]
By the elementary theory of ordinary differential equations, \(\phi(\zeta;\kappa)\) depends
analytically on \(\kappa\). We define the characteristic function
\[
\mathcal{D}(\kappa):=
-\frac{a(1)^3}{\beta_0}\phi_\zeta(1;\kappa)
+\frac{\alpha_0}{\beta_0}\phi(1;\kappa)
-\kappa^2\phi(1;\kappa).
\]
Then
\[
\kappa\in \sigma(\mathcal L)
\qquad\Longleftrightarrow\qquad
\mathcal{D}(\kappa)=0.
\]
Since \(\mathcal{D}\) is analytic and not identically zero, its zeros are isolated. Hence, the spectrum
of \(\mathcal L\) consists only of isolated eigenvalues.

Next, we show geometric simplicity. The scalar equation
\eqref{eq:scalar-eig-problem-fplane} is a second-order ordinary differential equation.
Imposing the boundary condition \(\phi(0)=0\) reduces the space of solutions to a
one-dimensional family, and therefore, any nontrivial solution of
\eqref{eq:scalar-eig-problem-fplane}--\eqref{eq:scalar-eig-bc-fplane} is unique up to
scalar multiplication. Hence, the scalar eigenspace is one-dimensional. Since \(\rho\), \(\tau\),
and \(\mathfrak{w}\) are recovered from \(\phi\) by the formulas
\[
\rho=\phi(1),
\qquad
\tau=-\kappa a(1)\phi(1),
\qquad
\mathfrak{w}=-\kappa a(\zeta)\phi,
\]
the eigenspace of \(\mathcal L\) is also one-dimensional. Thus, every eigenvalue of
\(\mathcal L\) is geometrically simple.

Finally, the algebraic multiplicity of an eigenvalue \(\kappa\) is finite because it is given by
the order of \(\kappa\) as a zero of the analytic characteristic function \(\mathcal{D}\). Since isolated
zeros of analytic functions have finite order, every eigenvalue of \(\mathcal L\) has finite
algebraic multiplicity. We conclude that the spectrum of \(\mathcal L\) consists of isolated, geometrically simple
eigenvalues of finite algebraic multiplicity.
\end{proof}
Using this proposition, we can decompose $\tilde{X}$ into two subspaces $\tilde{X}_1$ and $\tilde{X}_2$. $\tilde{X}_1$ contains eigenfunctions whose eigenvalues are all pure imaginary. Whereas ${X}_2$ is the space such that the spectrum of $\mathcal{L}$ restricted to $\tilde{X}_2$, we call it $\mathcal{L}_2$, it contains no pure imaginary eigenvalues. This shows that $\mathcal{L}$ satisfies $(A1)$ in theorem~\ref{CMT}. Next, we show that the second hypothesis $(A2)$ in the center manifold reduction theorem  is also satisfied. To achieve this, we take advantage of the following lemma.
\begin{lemma}\label{lem:resolvent-fplane}
There exist constants \(C>0\) and \(\xi_0>0\) such that each solution
\[
u\in D(\mathcal L)
\]
of the resolvent equation
\[
(\mathcal L-i\xi I)u=f^\dagger,
\qquad
f^\dagger\in \tilde{X},
\qquad
|\xi|>\xi_0,
\]
satisfies the estimates
\[
\|u\|_{\tilde{Y}}\le C\|f^\dagger\|_{\tilde{X}},
\qquad
\|u\|_{\tilde{X}}\le \frac{C}{|\xi|}\|f^\dagger\|_{\tilde{X}}.
\]
\end{lemma}

\begin{proof}
Let
\[
u=(\rho,\tau,\phi,\mathfrak{w}),
\qquad
f^\dagger=(\rho^\dagger,\tau^\dagger,\phi^\dagger,\mathfrak{w}^\dagger).
\]
Writing the resolvent equation
\[
(\mathcal L-i\xi I)u=f^\dagger
\]
componentwise, we obtain
\begin{align}
-a(1)^{-1}\tau-i\xi\rho&=\rho^\dagger,\label{eq:res1}\\
a(1)^4\beta_0^{-1}\phi_\zeta(1)-a(1)\alpha_0\beta_0^{-1}\rho-i\xi\tau&=\tau^\dagger,\label{eq:res2}\\
-a(\zeta)^{-1}\mathfrak{w}-i\xi\phi&=\phi^\dagger,\label{eq:res3}\\
\bigl(a(\zeta)^3\phi_\zeta\bigr)_\zeta-i\xi \mathfrak{w}&=\mathfrak{w}^\dagger,\label{eq:res4}
\end{align}
together with the boundary conditions
\[
\phi(0)=0,
\qquad
\mathfrak{w}(0)=0,
\qquad
\mathfrak{w}(1)=\tau.
\]

The argument follows the same energy method as in Groves--Wahl\'en \cite{GrovesWahlen2007}. We first differentiate
\eqref{eq:res3} with respect to \(\zeta\) and multiply by \(a(\zeta)^{3/2}\), obtaining
\[
-a(\zeta)^{3/2}\bigl(a(\zeta)^{-1}\mathfrak{w}\bigr)_\zeta-i\xi a(\zeta)^{3/2}\phi_\zeta
=
a(\zeta)^{3/2}\phi^\dagger_\zeta.
\]
Next, we multiply \eqref{eq:res4} by \(a(\zeta)^{-1/2}\), so that
\[
a(\zeta)^{-1/2}\bigl(a(\zeta)^3\phi_\zeta\bigr)_\zeta-i\xi a(\zeta)^{-1/2}\mathfrak{w}
=
a(\zeta)^{-1/2}\mathfrak{w}^\dagger.
\]
Squaring the moduli of these two equations and adding, we arrive at
\begin{align}
&a(\zeta)^3\bigl|\bigl(a(\zeta)^{-1}\mathfrak{w}\bigr)_\zeta\bigr|^2
+a(\zeta)^{-1}\bigl|\bigl(a(\zeta)^3\phi_\zeta\bigr)_\zeta\bigr|^2
+\xi^2\Bigl(a(\zeta)^3|\phi_\zeta|^2+a(\zeta)^{-1}|\mathfrak{w}|^2\Bigr)
\notag\\
&\qquad
+2\xi\,\Im\!\left(
\bigl(a(\zeta)^{-1}\mathfrak{w}\bigr)_\zeta\,a(\zeta)^3\overline{\phi_\zeta}
+
a(\zeta)^{-1}\overline{\mathfrak{w}}\,\bigl(a(\zeta)^3\phi_\zeta\bigr)_\zeta
\right)
\notag\\
&\qquad\qquad
=
a(\zeta)^3|\phi^\dagger_\zeta|^2+a(\zeta)^{-1}|\mathfrak{w}^\dagger|^2.
\label{eq:energy-pointwise}
\end{align}

Integrating \eqref{eq:energy-pointwise} over \(\zeta\in(0,1)\), the mixed term can be computed via
integration by parts:
\begin{align*}
&\int_0^1
\left(
\bigl(a^{-1}\mathfrak{w}\bigr)_\zeta\,a^3\overline{\phi_\zeta}
+
a^{-1}\overline{\mathfrak{w}}\,\bigl(a^3\phi_\zeta\bigr)_\zeta
\right)\,\dd \zeta
\\
&\qquad
=
\Bigl[a(\zeta)^2\mathfrak{w}\,\overline{\phi_\zeta}\Bigr]_{\zeta=0}^{\zeta=1}
=
a(1)^2\tau\,\overline{\phi_\zeta(1)},
\end{align*}
where we have used the fact that \(\mathfrak{w}(0)=0\). Hence, we get
\begin{align}
&\int_0^1
\left(
a^3\bigl|\bigl(a^{-1}\mathfrak{w}\bigr)_\zeta\bigr|^2
+
a^{-1}\bigl|\bigl(a^3\phi_\zeta\bigr)_\zeta\bigr|^2
\right)\,\dd \zeta
+
\xi^2
\int_0^1
\left(
a^3|\phi_\zeta|^2+a^{-1}|\mathfrak{w}|^2
\right)\,\dd \zeta
\notag\\
&\qquad
+
2\xi\,\Im\!\bigl(a(1)^2\tau\,\overline{\phi_\zeta(1)}\bigr)
=
\int_0^1
\left(
a^3|\phi^\dagger_\zeta|^2+a^{-1}|\mathfrak{w}^\dagger|^2
\right)\,\dd \zeta.
\label{eq:energy-integrated}
\end{align}

We next use the boundary equation \eqref{eq:res2}. Multiplying \eqref{eq:res2} by
\(\overline{\tau}\) and taking imaginary parts gives
\[
a(1)^4\beta_0^{-1}\Im\!\bigl(\phi_\zeta(1)\overline{\tau}\bigr)
-a(1)\alpha_0\beta_0^{-1}\Im\!\bigl(\rho\overline{\tau}\bigr)
-\xi|\tau|^2
=
\Im\!\bigl(\tau^\dagger\overline{\tau}\bigr).
\]
Equivalently,
\[
\Im\!\bigl(a(1)^2\tau\,\overline{\phi_\zeta(1)}\bigr)
=
\beta_0 a(1)^{-2}\xi|\tau|^2
+\alpha_0 a(1)^{-1}\Im\!\bigl(\rho\overline{\tau}\bigr)
+\beta_0 a(1)^{-2}\Im\!\bigl(\tau^\dagger\overline{\tau}\bigr).
\]
Inserting this identity into \eqref{eq:energy-integrated} yields
\begin{align}
&\int_0^1
\left(
a^3\bigl|\bigl(a^{-1}\mathfrak{w}\bigr)_\zeta\bigr|^2
+
a^{-1}\bigl|\bigl(a^3\phi_\zeta\bigr)_\zeta\bigr|^2
\right)\,\dd \zeta
+
\xi^2
\int_0^1
\left(
a^3|\phi_\zeta|^2+a^{-1}|\mathfrak{w}|^2
\right)\,\dd \zeta
\notag\\
&\qquad
+
2\beta_0 a(1)^{-2}\xi^2|\tau|^2
+
2\alpha_0 a(1)^{-1}\xi\,\Im\!\bigl(\rho\overline{\tau}\bigr)
\notag\\
&\le
C\|f^\dagger\|_{\tilde{X}}^2
+
C|\xi|\,|\tau|\,\|f^\dagger\|_{\tilde{X}},
\label{eq:energy-before-absorb}
\end{align}
for some constant \(C>0\).

It remains to control the mixed term involving \(\rho\). From \eqref{eq:res1}, we have
\[
-i\xi\rho=\rho^\dagger+a(1)^{-1}\tau,
\]
and therefore
\[
|\xi|\,|\rho|
\le
C\bigl(\|f^\dagger\|_{\tilde{X}}+|\tau|\bigr).
\]
Using this estimate in \eqref{eq:energy-before-absorb} and applying Young's inequality, we
find that for \(|\xi|\) sufficiently large, say \(|\xi|>\xi_0\), the mixed term can be absorbed
into the positive \(\xi^2|\tau|^2\)-contribution. Hence, there exists \(C>0\) such that
\begin{align}
&\int_0^1
\left(
a^3\bigl|\bigl(a^{-1}\mathfrak{w}\bigr)_\zeta\bigr|^2
+
a^{-1}\bigl|\bigl(a^3\phi_\zeta\bigr)_\zeta\bigr|^2
\right)\,\dd \zeta
\notag\\
&\qquad
+
\xi^2
\left(
|\tau|^2+\|\phi_\zeta\|_{L^2(0,1)}^2+\|\mathfrak{w}\|_{L^2(0,1)}^2
\right)
\le
C\|f^\dagger\|_{\tilde{X}}^2.
\label{eq:main-coercive}
\end{align}
We now derive the claimed \(\tilde{Y}\)- and \(\tilde{X}\)-norm bounds.

First, equation \eqref{eq:main-coercive} immediately gives
\[
|\tau|+\|\phi_\zeta\|_{L^2(0,1)}+\|\mathfrak{w}\|_{L^2(0,1)}
\le
\frac{C}{|\xi|}\|f^\dagger\|_{\tilde{X}}.
\]
Since \(\phi(0)=0\), Poincar\'e's inequality allows us to infer
\[
\|\phi\|_{L^2(0,1)}
\le
C\|\phi_\zeta\|_{L^2(0,1)}
\le
\frac{C}{|\xi|}\|f^\dagger\|_{\tilde{X}}.
\]
Similarly, equation \eqref{eq:res1} gives
\[
|\rho|
\le
\frac{C}{|\xi|}\|f^\dagger\|_{\tilde{X}}.
\]
Therefore, we obtain
\[
\|u\|_{\tilde{X}}\le \frac{C}{|\xi|}\|f^\dagger\|_{\tilde{X}}.
\]

In order to estimate the stronger \(\tilde{Y}\)-norm, we use \eqref{eq:res4} and obtain $(a^3\phi_\zeta)_\zeta=\mathfrak{w}^\dagger+i\xi\mathfrak{w}.$
Because \(\|\mathfrak{w}\|_{L^2}\le C|\xi|^{-1}\|f^\dagger\|_{\tilde{X}}\), this implies
\[
\|(a^3\phi_\zeta)_\zeta\|_{L^2(0,1)}
\le
C\|f^\dagger\|_{\tilde{X}}.
\]
Since \(a\) is smooth and strictly positive, standard elliptic regularity for this one-dimensional
boundary-value problem yields
\[
\|\phi\|_{H^2(0,1)}\le C\|f^\dagger\|_{\tilde{X}}.
\]
Similarly, from \eqref{eq:res3}, one has $-a^{-1}\mathfrak{w}=\phi^\dagger+i\xi\phi,$ and differentiating gives $-(a^{-1}\mathfrak{w})_\zeta=\phi^\dagger_\zeta+i\xi\phi_\zeta.$
Using the previously established bounds on \(\phi\) and \(\phi_\zeta\), we conclude that
\[
\|\mathfrak{w}\|_{H^1(0,1)}\le C\|f^\dagger\|_{\tilde{X}}.
\]
Combining these estimates with the bound for \(\rho\) and \(\tau\), we obtain
\[
\|u\|_{\tilde{Y}}\le C\|f^\dagger\|_{\tilde{X}}.
\]
This completes the proof.
\end{proof}

Observe that via $(A1)$ for $|\xi|\leq\xi_0$, $i\xi \notin \sigma(\mathcal{L}_2)$. This implies that $(\mathcal{L}_2-i\xi\mathcal{I})^{-1}$ exists for all $\xi$. Moreover, since the set $\{\xi:\xi\leq\xi_0\}$ is bounded, therefore,
\[
\norm{(\mathcal{L}_2-i\xi\mathcal{I})^{-1}}\leq C\leq C\frac{1+\xi_0}{1+\xi}\leq \frac{C_1}{1+\xi}.
\]
Via Lemma~\ref{lem:resolvent-fplane}, taking $\xi_0\geq 1$,
we can easily see that 
\[
\norm{(\mathcal{L}_2-i\xi\mathcal{I})^{-1}}\leq \frac{2C}{2|\xi|}\leq \frac{C_2}{1+|\xi|}.
\]
Thus, both cases $|\xi|\leq\xi_0$ and $|\xi|>\xi_0$ combined shows that $\mathcal{L}_2$ satisfies $(A2)$.

\subsection{Eigenvalues of the linearized problem}

For what is to follow, we write
\[
\nu=\kappa^2.
\] Our goal here is to recast the eigenvalue problem $\mathcal{L}u=\kappa u$ as a Sturm--Liouville problem. To that end, we introduce the Liouville transformation
\begin{equation}\label{eq:liouville-transform}
y=\int_0^\zeta a(s)^{-1}\,\dd s,
\qquad
v(y)=a(\zeta)\phi(\zeta).
\end{equation}
\begin{proposition}\label{prop:SL-reduction}
Under the change of variables \eqref{eq:liouville-transform}, the eigenvalue problem
for $\mathcal{L}$ is transformed into the Sturm--Liouville problem
\begin{equation}\label{eq:SL-problem}
\begin{aligned}
-&\,v_{yy}+\mathfrak{Q}(y)v=\nu v,
\qquad 0<y<1,\\&
v(0)=0,
\qquad
\frac{v_y(1)}{v(1)}=\widehat{\alpha}-\widehat{\beta}\nu.
\end{aligned}
\end{equation}
Here, $\mathfrak{Q}$ is given by
\begin{equation}\label{eq:Q-def}
\mathfrak{Q}=(aa_\zeta)_\zeta=\gamma'(\zeta(y)),
\end{equation} while $\widehat{\alpha}$ and $\widehat{\beta}$
are given by \eqref{eq:hat-params}.
\end{proposition}
\begin{proof}
Recall that in terms of $\nu$, equations \eqref{eq:phi-eigenvalue-problem}--\eqref{eq:phi-eigenvalue-bc} become
\begin{equation}\label{eq:phi-nu-problem}
-\,a(\zeta)^{-1}\bigl(a(\zeta)^3\phi_\zeta\bigr)_\zeta=\nu\phi,
\qquad 0<\zeta<1,
\end{equation}
\begin{equation}\label{eq:phi-nu-bc}
\phi(0)=0,
\qquad
-\frac{a(1)^3}{\beta_0}\phi_\zeta(1)+\frac{\alpha_0}{\beta_0}\phi(1)=\nu\phi(1).
\end{equation}
Thus, the spectral analysis of $\mathcal{L}$ reduces to the study of the scalar problem
\eqref{eq:phi-nu-problem}--\eqref{eq:phi-nu-bc}. 
Since
\[
\int_0^1 a(s)^{-1}\,\dd s=H(1)=1,
\]
the new variable $y$ in \eqref{eq:liouville-transform} ranges over the interval $(0,1)$. A direct calculation shows that
\[
\phi=\frac{v}{a},
\qquad
a^3\phi_\zeta=a\bigl(v_y-a_\zeta v\bigr),
\]
and hence, $
\bigl(a^3\phi_\zeta\bigr)_\zeta=v_{yy}-(aa_\zeta)_\zeta\,v.$
Therefore, equation \eqref{eq:phi-nu-problem} becomes
\[
-\,v_{yy}+\mathfrak{Q} v=\nu v,
\qquad 0<y<1,
\]
where
$\mathfrak{Q}$ is given by \eqref{eq:Q-def}.
Further, the boundary condition at $\zeta=0$ becomes $v(0)=0.$
At $\zeta=1$, using
\[
a(1)^3\phi_\zeta(1)=a(1)\bigl(v_y(1)-a_\zeta(1)v(1)\bigr),
\]
the boundary condition \eqref{eq:phi-nu-bc} reads
\[
a(1)\bigl(v_y(1)-a_\zeta(1)v(1)\bigr)
=
\bigl(\alpha_0-\beta_0\nu\bigr)\frac{v(1)}{a(1)}.
\]
Equivalently, this can be expressed as
\begin{equation}\label{eq:SL-bc}
\frac{v_y(1)}{v(1)}=\widehat{\alpha}-\widehat{\beta}\nu,
\end{equation}
where
\begin{equation}\label{eq:hat-params}
\widehat{\alpha}:=a_\zeta(1)+a(1)^{-2}\alpha_0,
\qquad
\widehat{\beta}:=a(1)^{-2}\beta_0.
\end{equation}
The proof is then complete. 
\end{proof}

The  Sturm--Liouville problem \eqref{eq:SL-problem} is an
equation with an eigenvalue-dependent affine boundary condition at $y=1$. The problem is known to have a countable number of geometrically simple eigenvalues $\nu_n$ for $n=0,1,2,...$ where Re$\nu_n<$Re$\nu_{n+1}$. They occur in complex-conjugate pairs and have the following representation
\[
\nu_n=n^2\pi^2+\int_{0}^1 \mathfrak{Q}(y)\;\dd y-\frac{2\widehat{\alpha}}{\widehat{\beta}}+O\left(\frac{1}{n}\right).
\]
The real eigenvalues $\nu$ of \eqref{eq:SL-problem}  may be characterized geometrically as the intersections of the line
\[
s=\widehat{\alpha}-\widehat{\beta}\nu
\]
with the graph
\[
B(\nu):=\frac{v_y(1;\nu)}{v(1;\nu)},
\]
where $v(\cdot;\nu)$ denotes the solution of the initial-value problem
\[
-\,v_{yy}+\mathfrak{Q}(y)v=\nu v,
\qquad
v(0;\nu)=0.
\]
A tangential intersection between the line $s$ and graph $B$ corresponds to an eigenvalue of algebraic multiplicity two.

Let $\nu_n^D$ denote the Dirichlet eigenvalues of the self-adjoint problem
\[
-\,v_{yy}+\mathfrak{Q}(y)v=\nu v,
\qquad
v(0)=v(1)=0,
\]
where $\nu_n^D$ obeys the following asymptotic representation
\[
\nu_n^D=(n+1)^2\pi^2+\int_0^1 \mathfrak{Q}(y)\;\dd y+O\left(\frac{1}{n}\right).
\]
Then, $B(\nu)$ has poles precisely at the values $\nu_n^D$ and is strictly decreasing
from $+\infty$ to $-\infty$ on each interval
\[
(-\infty,\nu_0^D),
\qquad
(\nu_{n-1}^D,\nu_n^D),\quad n\in\N.
\]
It follows that the boundary-value problem
\eqref{eq:SL-problem}
has at least one real eigenvalue in each interval
\[
(\nu_{n-1}^D,\nu_n^D),\qquad n\in\N.
\] 
Additionally, $\nu_{n-1}^D<\text{Re} \nu_{n+1}<\nu_n^D$ for $n$ large enough.
Consequently, the operator $\mathcal{L}$ has at least one pair of real eigenvalues in each interval
\[
\bigl((\nu_{n-1}^D)^{1/2},(\nu_n^D)^{1/2}\bigr)
\quad \text{and} \quad
\bigl(-(\nu_n^D)^{1/2},-(\nu_{n-1}^D)^{1/2}\bigr),
\qquad n\in\N.
\]

The distinguished low-frequency spectrum of $\mathcal{L}$ is therefore governed by the behavior
of the first few eigenvalues of \eqref{eq:SL-problem},
and the corresponding bifurcation curves in parameter space in  Figure~\ref{Bifurcation curve}  arise when the line
\[
s=\widehat{\alpha}-\widehat{\beta}\nu
\]
has a multiple intersection with the graph $s=B(\nu)$, or when one of the lowest
eigenvalues $\nu$ crosses zero. In particular, since $\nu=\kappa^2$, the sign of $\nu$
determines the type of spatial eigenvalue:
\begin{itemize}
    \item if $\nu>0$, then $\kappa=\pm\sqrt{\nu}$ are real spatial eigenvalues;
    \item if $\nu<0$, then $\kappa=\pm i\sqrt{|\nu|}$ are purely imaginary spatial eigenvalues.
\end{itemize}
Thus, the appearance of small-amplitude solitary waves is controlled by the sign changes
and multiplicity changes of the lowest eigenvalues of the Sturm--Liouville problem
\eqref{eq:SL-problem}.

Observe that a zero eigenvalue of \eqref{eq:SL-problem} gives rise to a zero eigenvalue of the operator $\mathcal{L}.$ The algebraic multiplicity of this eigenvalue can be determined by investigating $\mathcal{L}u=0.$ Setting $\nu=0$ in \eqref{eq:phi-nu-problem} and \eqref{eq:phi-nu-bc} yields
\begin{equation}\label{eq:zero-ode}
\bigl(a(\zeta)^3\phi_\zeta\bigr)_\zeta=0,
\qquad
\phi(0)=0,
\qquad
-\frac{a(1)^3}{\beta_0}\phi_\zeta(1)+\frac{\alpha_0}{\beta_0}\phi(1)=0.
\end{equation}
This leads to 
\[
\phi(\zeta)=C\int_0^\zeta a(s)^{-3}\,\dd s=:CI(\zeta).
\]
Then
\[
\phi(1)=CI(1),
\qquad
\phi_\zeta(1)=Ca(1)^{-3}.
\]
Inserting these two facts into the boundary condition in \eqref{eq:zero-ode} gives
\[
-\frac{a(1)^3}{\beta_0}\cdot Ca(1)^{-3}+\frac{\alpha_0}{\beta_0}\cdot CI(1)=0,
\]
which is equivalent to saying
\[
-C+\alpha_0 I(1)\,C=0.
\]
Hence, a nontrivial solution exists if and only if
\begin{equation}\label{eq:alpha-star-fplane}
\alpha_0=\alpha_*:=I(1)^{-1}
=
\left(\int_0^1 a(s)^{-3}\,\dd s\right)^{-1}.
\end{equation}
Next, we show that when $\alpha_0=\alpha_\ast$, we are guaranteed that $0$ is an eigenvalue of $\mathcal{L}$ and has algebraic multiplicity 2. Additionally, as shown below, if $\beta_0$ is set to equal some number $\beta_\ast$, then $0$ becomes an eigenvalue of $\mathcal{L}$ with algebraic multiplicity 4.
\begin{figure}[H]
\centering
\begin{tikzpicture}[
    x=3.0cm,
    y=3.0cm,
    >=Latex,
    every node/.style={font=\small},
    callout/.style={rounded corners=2pt, inner sep=4pt, line width=0.5pt},
]

% ---------------------------------------------------------------------
% Schematic coordinate window
% ---------------------------------------------------------------------
\def\xmin{-1.38}
\def\xmax{3.45}
\def\ymin{-0.093}
\def\ymax{2.34}
\def\alphastar{1.00}
\def\betastar{1.00}

% % Background regions.
% \fill[blue!7]    (\xmin,\alphastar) rectangle (\xmax,\ymax);
% \fill[yellow!13] (\xmin,\ymin)      rectangle (\xmax,\alphastar);
% \draw[gray!25,line width=0.6pt] (\xmin,\ymin) rectangle (\xmax,\ymax);

% Reference lines.
\draw[gray!65,dash pattern=on 2pt off 4pt,line width=0.65pt]
    (\xmin,\alphastar) -- (\xmax,\alphastar);
\draw[gray!55,dash pattern=on 1.5pt off 4pt,line width=0.55pt]
    (\betastar,\ymin) -- (\betastar,\alphastar);
% \draw[graytext,dash pattern=on 6pt off 3pt on 1pt off 3pt,line width=0.6pt]
%     (\xmin,0) -- (\xmax,0);

% Axes.
\draw[axiscolor,->,line width=0.8pt]
    ({\xmin+0.03},{\ymin+0.06}) -- ({\xmax-0.05},{\ymin+0.06});
\draw[axiscolor,->,line width=0.8pt]
    ({\xmin+0.03},{\ymin+0.06}) -- ({\xmin+0.03},{\ymax-0.06});

\node[axiscolor,anchor=west,font=\large]
    at ({\xmax-0.02},{\ymin+0.08}) {$\beta$};
\node[axiscolor,anchor=south west,font=\large]
    at ({\xmin+0.01},{\ymax-0.08}) {$\alpha$};

% ---------------------------------------------------------------------
% Bifurcation curves
% ---------------------------------------------------------------------

% C_2: Hamiltonian--Hopf branch, shifted 2 units right.
\begin{scope}[shift={(1,0)}]
\draw[tealcurve,line width=1.45pt]
    plot[domain=-0.82:0,samples=120,smooth]
    (\x,{1 + 1.05*(-\x/0.82)^2 + 0.12*(-\x/0.82)^3});
\end{scope}

% C_1: real 1:1 branch.
\begin{scope}[shift={(1,0)}]
\draw[redcurve,line width=1.35pt,dash pattern=on 8pt off 5pt]
    plot[domain=0:1.02,samples=120,smooth]
    (\x,{1 + 0.15*(\x)/1.02)
          + 0.40*(\x/1.02)^2
          + 0.62*(\x/1.02)^3});
\end{scope}

% C_4 and C_3: zero-eigenvalue branches.
\draw[axiscolor,line width=1.35pt] (1,\alphastar) -- (3.3,\alphastar);
%\draw[gray!65,line width=0.8pt,dash pattern=on 2pt off 5pt]
    %({\xmin+1},\alphastar) -- (1,\alphastar);

% Codimension-two point.
\draw[axiscolor,line width=0.9pt,fill=white]
    (\betastar,\alphastar) circle[radius=0.040];
\fill[axiscolor] (\betastar,\alphastar) circle[radius=0.014];
\draw[gray!25,line width=1.1pt]
    (\betastar,\alphastar) circle[radius=0.057];

% Curve labels.
\node[fill=white,inner sep=1pt,text=tealcurve,font=\large]
    at (0.2,1.70) {$C_2$};
\node[fill=white,inner sep=1pt,text=redcurve,font=\large]
    at (1.6,1.55) {$C_1$};
\node[fill=white,inner sep=1pt,text=axiscolor,font=\large]
    at (1.9,1.1) {$C_4$};
\node[fill=white,inner sep=1pt,text=gray!70,font=\normalsize]
    at (-1.20,1.06) {$C_3$};

% Axis and reference labels.
\node[anchor=east,font=\scriptsize]
    at ({\xmin-0.03},{\alphastar-0.02}) {$\alpha_*$};
\node[anchor=south,font=\scriptsize]
    at (\betastar,{\ymin+0.02}) {$\beta_*$};
\node[anchor=east,font=\scriptsize]
    at ({\xmin-0.03},0) {$0$};

% ---------------------------------------------------------------------
% Callout boxes
% ---------------------------------------------------------------------

\node[
    callout,
    draw=tealcurve,
    fill=green!7,
    text=tealcurve!45!black,
    align=left
]
    (HH) at (1,2.17)
    {Hamiltonian--Hopf\\[-1pt] $\nu=-\mathfrak{q}^2<0$};

    \draw[tealcurve,->,line width=0.65pt]
    (HH.south west) to[out=-100,in=22] (0.55,1.30);

\node[
    callout,
    draw=redcurve,
    fill=orange!8,
    text=redcurve!55!black,
    align=left
]
    (R) at (2.68,2.11)
    {real $1{:}1$ resonance\\[-1pt] $\nu=\kappa^2>0$};

\draw[redcurve,->,line width=0.65pt]
    (R.south west) to[out=-130,in=42] (1.6,1.30);

\node[
    callout,
    draw=axiscolor,
    fill=white,
    align=center,
    text=axiscolor,
    font=\tiny,
    text width=2.25cm
]
    (CT) at (1,0.57)
    {codimension-two point\\[-1pt]
     quadruple zero eigenvalue\\[-1pt]
     $(\beta_*,\alpha_*)$};

\draw[axiscolor,->,line width=0.6pt]
    (CT.north east) -- (1,1);

\node[
    callout,
    draw=axiscolor,
    fill=white,
    align=center,
    text=axiscolor,
    font=\scriptsize,
    text width=2.20cm
]
    (Z) at (2.86,0.57)
    {$0^2$ resonance\\[-1pt] double zero eigenvalue};

\draw[axiscolor,->,line width=0.6pt]
    (Z.north) to[out=90,in=-80] (1.5,1.00);

% Local schematic box.
\node[
    callout,
    draw=gray!35,
    fill=white,
    align=left,
    text=slate,
    font=\tiny,
    anchor=south west,
    text width=5.05cm
]
    at ({\xmin+0.11},0.08)
    {
     $C_1:\;(\beta,\alpha)=(\beta_*+O(\kappa^2),\alpha_*+O(\kappa^4))$\\[-1pt]
     $C_2:\;(\beta,\alpha)=(\beta_*-O(\mathfrak{q}^2),\alpha_*+O(\mathfrak{q}^4))$\\[-1pt]
     $C_4:\;(\beta,\alpha)=(\beta,\alpha_*),\quad \beta>\beta_*$};

% Weak effective gravity box.
\node[
    align=right,
    text=graytext,
    font=\tiny,
    anchor=north east
]
    at ({\xmax-0.02},0.2)
    {weak effective gravity: $\alpha=0$\\[-1pt]
     $g^*=g-2\Omega c=0$};
\end{tikzpicture}
\caption{Local bifurcation diagram in the $(\beta,\alpha)$-plane for the steady $f$-plane problem.
The point $(\beta_\ast,\alpha_\ast)$ is a codimension-two point where the linearized operator has a
quadruple zero eigenvalue. The branch $C_4$ corresponds to the $0^2$ resonance (double zero eigenvalue),
$C_1$ to the real $1:1$ resonance, and $C_2$ to the Hamiltonian--Hopf bifurcation.}
\label{Bifurcation curve}
\end{figure}
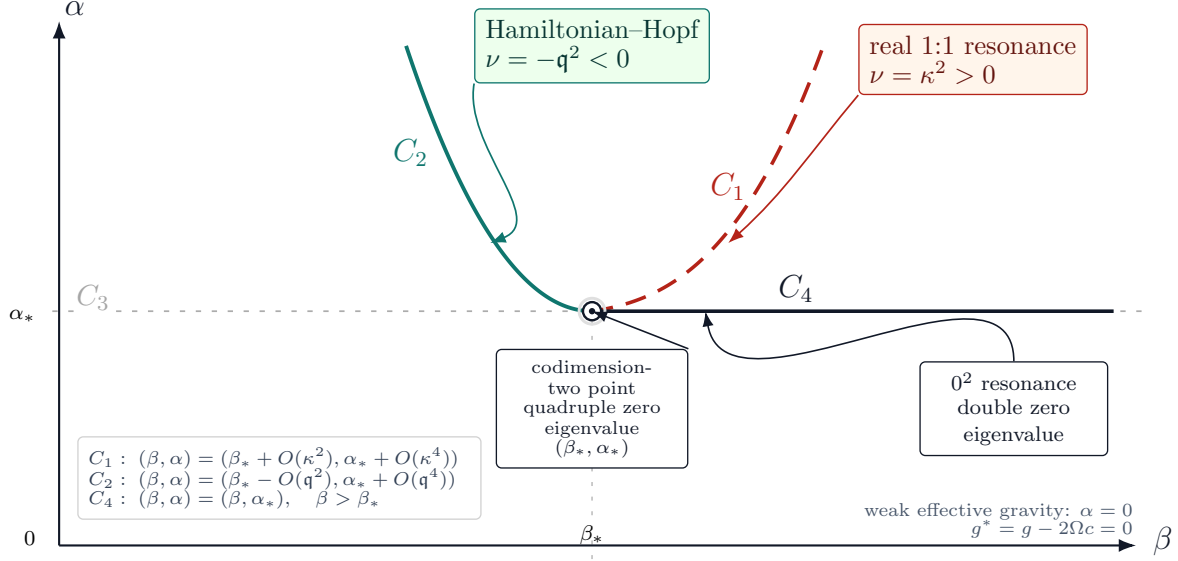

\subsubsection{Eigenvalue $0$ with multiplicity 2}\label{double} Assume from now on that $\alpha_0=\alpha_*$.
We define
\begin{equation}\label{eq:phi1-def}
\phi_1(\zeta):=\alpha_*I(\zeta).
\end{equation}
Then
\[
\phi_1(0)=0,
\qquad
\phi_1(1)=1,
\qquad
a(1)^3\phi_1'(1)=\alpha_*.
\]
It follows that
\[
\mathcal{L}w_1=0,
\]
where
\begin{equation}\label{eq:w1-def}
w_1:=
\begin{pmatrix}
1\\[0.3ex]
0\\[0.3ex]
\phi_1(\zeta)\\[0.3ex]
0
\end{pmatrix}
\in D(\mathcal{L}).
\end{equation}

To obtain a Jordan chain of length two, we seek $w_2=(\rho_2,\tau_2,\phi_2,\mathfrak{w}_2)$ such that
\[
\mathcal{L}w_2=w_1.
\]
Writing out the components explicitly gives
\begin{align}
-a(1)^{-1}\tau_2 &= 1,\label{eq:w2-1}\\
a(1)^4\beta_0^{-1}\phi_2'(1)-a(1)\alpha_*\beta_0^{-1}\rho_2&=0,\label{eq:w2-2}\\
-a(\zeta)^{-1}\mathfrak{w}_2&=\phi_1,\label{eq:w2-3}\\
\bigl(a(\zeta)^3\phi_2'(\zeta)\bigr)'&=0,\label{eq:w2-4}
\end{align}
subject to constraints
\[
\phi_2(0)=0,
\qquad
\mathfrak{w}_2(0)=0,
\qquad
\mathfrak{w}_2(1)=\tau_2.
\]
From \eqref{eq:w2-1} and \eqref{eq:w2-3}, we obtain
$\tau_2=-a(1)$ and $\mathfrak{w}_2(\zeta)=-a(\zeta)\phi_1(\zeta)$, respectively.
Since $\phi_1(1)=1$, the boundary condition $\mathfrak{w}_2(1)=\tau_2$ is automatically satisfied. Further, equation \eqref{eq:w2-4} again implies that $\phi_2$ is of the form
\[
\phi_2(\zeta)=C\int_0^\zeta a(s)^{-3}\,\dd s.
\]
Substituting into \eqref{eq:w2-2} yields
\[
a(1)^4\beta_0^{-1}\cdot Ca(1)^{-3}-a(1)\alpha_*\beta_0^{-1}\cdot CI(1)=0,
\]
which is identically satisfied because $\alpha_*I(1)=1$. Thus one may choose $C=0$, so that
\[
\phi_2\equiv 0,\qquad \rho_2=0.
\]
Hence, a convenient generalized eigenvector is
\begin{equation}\label{eq:w2-def}
w_2:=
\begin{pmatrix}
0\\[0.3ex]
-a(1)\\[0.3ex]
0\\[0.3ex]
-a(\zeta)\phi_1(\zeta)
\end{pmatrix},
\end{equation}
and we have the Jordan chain
\begin{equation}\label{eq:jordan-chain-02}
\mathcal{L}w_1=0,
\qquad
\mathcal{L}w_2=w_1.
\end{equation}
\subsubsection{Eigenvalue $0$ with multiplicity 4}\label{quadruple}
After constructing $w_1$ and $w_2$ so that
\[
\mathcal{L}w_1=0,
\qquad
\mathcal{L}w_2=w_1,
\]
one seeks a further generalized eigenvector $w_3$ satisfying
\[
\mathcal{L}w_3=w_2.
\]
It is precisely the solvability condition for this equation that determines a new parameter, $\beta_\ast,$ that will allow 0 to have algebraic multiplicity 4. 
Hence,  as before the first two vectors in the Jordan chain may be written as $w_1$ and $w_2$ in \eqref{eq:w1-def} and \eqref{eq:w2-def}, respectively.

To continue the chain, one looks for $w_3$ of the form
\[
w_3=(-\alpha_*\Theta(1),0,-\alpha_*\Theta(\zeta),0)^\text{T}, \quad \text{where }\Theta(\zeta):=\int_0^\zeta a(r)^{-3}\int_0^r a(t)I(t)\,\dd t\,\dd r.
\]
The equation $\mathcal{L}w_3=w_2$ is solvable if and only if the top boundary condition is satisfied.
Substituting this expression into the boundary condition yields the condition
\[
\beta_0=\alpha_*^2\int_0^1 a(\zeta)I(\zeta)^2\,\dd \zeta.
\]
This quantity is then defined to be $\beta_*$, namely
\begin{equation}\label{eq:beta-star-fplane}
\beta_*:=
\alpha_*^2\int_0^1 a(\zeta)\left(\int_0^\zeta a(s)^{-3}\,\dd s\right)^2\dd \zeta.
\end{equation}

The quantity $\beta_*$ is the extra condition that allows the Jordan chain
at zero to be extended further. In other words, we can conclude that
\[
\alpha_0=\alpha_*
\quad\Longrightarrow\quad
\text{zero has algebraic multiplicity 2},
\]
whereas
\[
\alpha_0=\alpha_*
\text{ and }
\beta_0=\beta_*
\quad\Longrightarrow\quad
\text{zero has algebraic multiplicity 4}.
\]
Finally, we seek \(w_4\) such that $\cL w_4=w_3.$
We define
\[
w_4:=(\rho_4,\tau_4,\phi_4(\zeta),\mathfrak{w}_4(\zeta))^{\text{T}}.
\]
Then, we obtain
\[
\tau_4=a(1)\alpha_*\Theta(1).
\]
Further, we also have
\[
\mathfrak{w}_4(\zeta)=a(\zeta)\alpha_*\Theta(\zeta).
\]
The remaining equations are
\[
a(1)^4\beta_*^{-1}\phi_4'(1)-a(1)\alpha_*\beta_*^{-1}\rho_4=0,
\qquad
\bigl(a(\zeta)^3\phi_4'(\zeta)\bigr)'=0.
\]
Standard computations imply
\[
\phi_4(\zeta)=C_2I(\zeta),
\qquad
\rho_4=\phi_4(1)=C_2I(1)=C_2\alpha_*^{-1},
\]
for some constant $C_2$.
As before, here, we may
choose \(C_2=0\). Therefore,
\[
\phi_4\equiv0,
\qquad
\rho_4=0.
\]
Hence, a convenient choice for $w_4$ is
\[
w_4=
\begin{pmatrix}
0\\[0.4ex]
a(1)\alpha_*\Theta(1)\\[0.4ex]
0\\[0.4ex]
a(\zeta)\alpha_*\Theta(\zeta)
\end{pmatrix}.
\]

\medskip

In summary, using the linearized operator \(\cL\), we obtain the Jordan chain
\[
\cL w_1=0,
\qquad
\cL w_2=w_1,
\qquad
\cL w_3=w_2,
\qquad
\cL w_4=w_3.
\]
The first two vectors give rise to the Hamiltonian \(0^2\) resonance, while the full
length-four chain occurs only at the codimension-two point \((\beta_*,\alpha_*)\) which gives rise to Hamiltonian real 1:1 which will be explained in full details in the next two sections.
%%%%%%%%%%%%%%%%%%%%%%%%%%%%%%%%%%%%%%%%%%%%%%%%
\section{Hamiltonian Resonances}\label{sec:Resonances}

\subsection{The Hamiltonian $0^2$ resonance}\label{sec:02-resonance}

In this section we formulate the $f$-plane analogue of the Hamiltonian $0^2$ resonance
treated by Groves and Wahl\'en \cite{GrovesWahlen2007}.
Using the parameter point $\alpha_0=\alpha_\ast$ found earlier, the linearized operator has
a double eigenvalue at the origin. We then construct the corresponding two-dimensional center space 
and derive the reduced Hamiltonian normal form whose homoclinic solutions generate
small-amplitude solitary waves. We have included the majority of tedious calculations in this section that are skipped in \cite{GrovesWahlen2007}.

% \subsection{The critical curve and the zero eigenvalue}
% We seek the analogue of the curve $C_4$ in the $(\beta,\alpha)$-plane. Observe that the capillary parameter $\beta_0$ does not enter the condition
% \eqref{eq:alpha-star-fplane}; it enters only at the next stage, when one asks for a generalized
% eigenvector and hence algebraic multiplicity two.

% \begin{remark}
% Formula \eqref{eq:alpha-star-fplane} is the direct counterpart of the condition
% $\alpha=\alpha_*$ in Groves--Wahl\'en. In the present $f$-plane setting, however, the
% effective gravity parameter is
% \[
% \alpha=\frac{(g-2\Omega c)d^3}{m^2},
% \]
% so fixing $\alpha_0=\alpha_*$ is equivalent to selecting a distinguished wave speed $c_0$.
% \end{remark}

Let $\Psi$ denote the constant symplectic form given in \eqref{eq:symplectic form flattened}.
A direct calculation using $w_1$ and $w_2$ in \eqref{eq:w1-def} and \eqref{eq:w2-def} yields
\begin{equation}
\begin{aligned}\label{02symplectic}
\Psi(w_1,w_2):=\Phi^{0}_0(w_1,w_2)
&
=
\int_0^1 \bigl(-a(\zeta)\phi_1(\zeta)\bigr)\phi_1(\zeta)\,\dd \zeta
-\frac{\beta_0}{a(1)}(-a(1))
\\&=
-\int_0^1 a(\zeta)\phi_1(\zeta)^2\,\dd \zeta+\beta_0\\
&=\beta_0-\alpha_*^2\int_0^1 a(\zeta)\left(\int_0^\zeta a(s)^{-3}\,\dd s\right)^2\dd \zeta\\
&=\beta_0-\beta_\ast.
\end{aligned}
\end{equation}
Using the quantity computed above, we define the normalization of vectors $w_1$ and $w_2$ as follows
\begin{equation}\label{eq:symplectic-basis-fplane}
e:=(\beta_0-\beta_\ast)^{-1/2}w_1,
\qquad
f:=(\beta_0-\beta_\ast)^{-1/2}w_2.
\end{equation}
As a result, we obtain
\begin{equation}\label{eq:symplectic-normalization}
\Psi(e,f)=1,
\qquad
\Psi(e,e)=\Psi(f,f)=0.
\end{equation}
Thus, the vectors $\{e,f\}$ form a symplectic basis for the two-dimensional generalized eigenspace
associated with the double zero eigenvalue. Hence, the center subspace of the center manifold is precisely given by
\[
X_c=\operatorname{span}\{e,f\}.
\]
Using Darboux coordinates on the center manifold, the full variable of an element in the center manifold is given by 
\[
u_1=\tilde{u}_1+\tilde{r}(\tilde{u}_1;\varepsilon)=q e+p f+\tilde{r}(\tilde{u}_1;\varepsilon),
\]
where $q=q(x)$ and $p=p(x)$ are the finite-dimensional center coordinates, $\tilde{u}_1$ is the center component, and $\tilde{r}$ is the hyperbolic component.
Let $\widetilde H^\eps(q,p)$ denote the reduced Hamiltonian on the center manifold, where
\[
\eps=(\eps_1,\eps_2),
\qquad
\beta=\beta_0+\eps_1,
\qquad
\alpha=\alpha_*+\eps_2.
\]
Taking advantage of the reduced Hamiltonian 
\[
\tilde{H}^\eps(\tilde{u}_1)=K^\eps(\tilde{u}_1+\tilde{r}(\tilde{u}_1;\varepsilon))
\] and expanding it around the equilibrium $(\tilde{u}_1,\varepsilon)=(0,0)$ yields
\begin{equation}\label{eq:H-expansion-0^2}
\tilde H^\varepsilon
=
\tilde H^{0,0}_2
+\varepsilon_1\tilde H^{1,0}_2
+\varepsilon_2\tilde H^{0,1}_2
+\varepsilon_1^2\tilde H^{2,0}_2
+\tilde H^{0,0}_3
+\cdots.
\end{equation}
The superscript $i,j$ in $\tilde{H}^{i,j}_{k}$ means the order in $\varepsilon_1,\varepsilon_2$, whereas the subscript $k$ means the degree in $q$ or $p$.
For the \(0^2\) resonance, we fix \(\varepsilon_1=0\) and write
\(\varepsilon_2=\delta\), where \(0<|\delta|\ll1\). The homoclinic branch
exists on the side
\[
    \frac{\delta}{\beta_0-\beta_*}>0.
\]
Thus, we move transversely to the critical curve by varying the effective gravity parameter
while keeping $\beta=\beta_0$ fixed. Hence, equation \eqref{eq:H-expansion-0^2} reads
\begin{equation}
\tilde H^{0,\delta}
=
\tilde H^{0,0}_2+
+\delta\tilde H^{0,1}_2
+\tilde H^{0,0}_3
+\cdots.
\end{equation}
To determine the pure quadratic part of the reduced Hamiltonian above, we need to take advantage of the identity 
\begin{equation}\label{eq:quadratic-identity}
\Psi(\mathcal{L}u_1^{(1)},u_1^{(2)})=2K^{0,0}_2[u_1^{(1)},u_1^{(2)}],
\end{equation}
where $K^{0,0}_2$ is the quadratic part of the Hamiltonian \eqref{eq:Keps-fplane} in the flattened variables.
Using \eqref{eq:jordan-chain-02} and \eqref{eq:symplectic-normalization}, one finds
\[
K^{0,0}_2[e,e]=0,
\qquad
K^{0,0}_2[e,f]=0,
\qquad
K^{0,0}_2[f,f]=\frac12.
\]
Therefore,
\begin{equation}\label{eq:H2-02}
\widetilde H^{0,0}_2(q,p)=\frac12 p^2.
\end{equation}
This is the signature of the Hamiltonian $0^2$ resonance: the quadratic Hamiltonian is
degenerate in the $q$-direction and contains only the momentum contribution.

The next nontrivial terms are the parameter-dependent quadratic correction and the cubic
term. More precisely, the reduced Hamiltonian takes the form
\begin{equation}\label{eq:reduced-H-02}
\widetilde H^{0,\delta}(q,p)
=
\frac12p^2+c_1\delta q^2+c_2q^3
+O\!\left(|p||(q,p)||(\delta,q,p)|\right)
+O\!\left(|(q,p)|^2|(\delta,q,p)|^2\right),
\end{equation}
where the coefficients $c_1$ and $c_2$ are determined by the Taylor expansion of the
Hamiltonian and the center-manifold reduction map. More precisely, using \eqref{eq:quadratic-identity}
\begin{equation}\label{eq:c1-def}
\begin{aligned}
c_1
&=
K^{0,1}_2[e,e]
+2K^{0,0}_2[e,r_{10}^{0,1}]\\&
=-\frac{1}{2(\beta_0-\beta_\ast)}<0,
\end{aligned}
\end{equation}
and
\begin{equation}\label{eq:c2-def}
\begin{aligned}
c_2
&=
K^{0,0}_3[e,e,e]
+2K^{0,0}_2[e,r_{20}^{0,0}]+K^{0,0}_3[e,e,r_{10}^{0,0}]\\&
=\frac{-1}{2(\beta_0-\beta_\ast)^{3/2}}(\alpha_\ast)^3 \int_0^1 a(\zeta)^{-5}\,\dd \zeta<0.
\end{aligned}
\end{equation}
The second term in the expression of $c_2$ above vanishes due to the fact that $\mathcal{L}e=0$. The third term vanishes because $\tilde{r}_{10}^{0,0}=0$ due to the condition $\tilde{r}(0;0)=0$ and $d_1\tilde{r}[0;0]=0.$
The superscript $i,j$ in $\tilde{r}^{i,j}_{kl}$ means the order in $\varepsilon_1,\varepsilon_2$, whereas the subscript $k,l$ means degree in $q,p$. These terms come from the Taylor expansion of
the reduction map. 

As a result, the associated Hamilton equations  generated by \eqref{eq:reduced-H-02} take the form
\begin{align}
q_x&=p+R_1(q,p,\delta),\label{eq:qeq-02}\\
p_x&=-2c_1\delta q-3c_2q^2+R_2(q,p,\delta),\label{eq:peq-02}
\end{align}
where
\[
R_1=O\!\left(|(q,p)||(\delta,q,p)|\right),
\qquad
R_2=O\!\left(|(q,p)||(\delta,q,p)|^2\right).
\]
We define
\[
\mathfrak{c}_1:=-2c_1>0,
\qquad
\mathfrak{c}_2:=-3c_2>0.
\]
Then \eqref{eq:qeq-02}--\eqref{eq:peq-02} become
\begin{align}
q_x&=p+R_1(q,p,\delta),\label{eq:qeq-02b}\\
p_x&=\mathfrak{c}_1\delta q+\mathfrak{c}_2 q^2+R_2(q,p,\delta).\label{eq:peq-02b}
\end{align}

We now introduce the following scalings
\begin{equation}\label{eq:02-scaling}
X=\sqrt{\mathfrak{c}_1\delta}\,x,
\qquad
q(x)=\frac{3\mathfrak{c}_1}{2\mathfrak{c}_2}\,\delta\,Q(X),
\qquad
p(x)=\frac{3\mathfrak{c}_1^{3/2}}{2\mathfrak{c}_2}\,\delta^{3/2}\,P(X).
\end{equation}
Substituting \eqref{eq:02-scaling} into \eqref{eq:qeq-02b}--\eqref{eq:peq-02b} yields
\begin{align}
Q_X&=P+\delta^{1/2}R_3(Q,P,\delta),\label{eq:Qeq}\\
P_X&=Q+\frac32Q^2+\delta^{1/2}R_4(Q,P,\delta),\label{eq:Peq}
\end{align}
where $R_3$ and $R_4$ are smooth and uniformly bounded on bounded subsets.
In the singular limit $\delta=0$, the reduced system is
\begin{equation}\label{eq:model-02}
Q_X=P,
\qquad
P_X=Q+\frac32Q^2.
\end{equation}
Eliminating $P$ gives
\begin{equation}\label{eq:scalar-model-02}
Q_{XX}=Q+\frac32Q^2.
\end{equation}
It is known that equation \eqref{eq:scalar-model-02} possesses the explicit homoclinic solution
\begin{equation}\label{eq:homoclinic-02}
Q(X)=-\operatorname{sech}^2\!\left(\frac X2\right),
\qquad
P(X)=Q'(X).
\end{equation}
The reduced system \eqref{eq:Qeq}--\eqref{eq:Peq} is reversible under
\[
S(Q,P)=(Q,-P).
\]

At $\delta=0$, the stable manifold of origin intersects the symmetric section
\[
\operatorname{Fix}S=\{P=0\}
\]
transversely at the point corresponding to the homoclinic orbit \eqref{eq:homoclinic-02}.
By smooth dependence of invariant manifolds on parameters, the intersection persists for all
sufficiently small $|\delta|$ on the side $\delta/(\beta_0-\beta_*)>0$. Hence, the reduced system possesses a symmetric homoclinic
solution in this parameter regime.
%%%%%%%%%%%%%%%%%%%%%%%%%%%%%%%%%%%%%%%%%%%%%%%%%%%%%%%%
\\

The previous normalization is convenient on the branch \(\beta_0>\beta_*\).
For the polarity statement, however, it is preferable to return to the
unnormalized Jordan basis \(\{w_1,w_2\}\), since this basis remains real on
both sides of \(\beta_*\). This allows us to treat the elevation and depression
branches simultaneously.
\begin{proof}[Proof of Corollary~\ref{cor:threshold}]\label{proof of cor}
Let
\[
    s:=\beta_0-\beta_\ast,
    \qquad
    \delta:=\alpha-\alpha_\ast,
\]
and define
\[
    C_\ast:=
    \alpha_\ast^3\int_0^1 a(\zeta)^{-5}\,\dd \zeta>0.
\]
We use the unnormalized center coordinates associated with the Jordan chain
\[
    \mathcal L w_1=0,
    \qquad
    \mathcal L w_2=w_1,
\]
and write, on the center manifold,
\[
    u(x)=A(x)w_1+B(x)w_2+\widetilde r(A,B;\delta),
\]
where the dependence on the vertical variable $\zeta$ is contained in the basis vectors and in the function-valued components of $\widetilde r$.
Since the first component of \(w_1\) is \(1\), while the first component of
\(w_2\) is \(0\), the free-surface displacement satisfies
\begin{equation}\label{eq:eta-A-leading}
    \eta(x)=\rho(x)
    =
    A(x)
    +O\bigl(|(A,B)|^2+|\delta||(A,B)|\bigr).
\end{equation}
Thus, the sign of the leading-order surface profile is determined by \(A\).

In these unnormalized coordinates, the restriction of the symplectic form to
\(\operatorname{span}\{w_1,w_2\}\) is
\[
    \Psi=s \,\dd A\wedge \dd B,
\]
because \(\Psi(w_1,w_2)=\beta_0-\beta_\ast=s\). The reduced Hamiltonian
computed in the Hamiltonian \(0^2\) normal form becomes
\begin{equation}\label{eq:H-unnormalized-02}
    \widetilde H(A,B;\delta)
    =
    \frac{s}{2}B^2
    -\frac{\delta}{2}A^2
    -\frac{C_\ast}{2}A^3
    +\textup{higher-order terms}.
\end{equation}
Indeed, when \(s>0\), this follows from the normalized variables
\(q=\sqrt s\,A\), \(p=\sqrt s\,B\), together with
\[
    c_1=-\frac{1}{2s},
    \qquad
    c_2=-\frac{C_\ast}{2s^{3/2}}.
\]
The unnormalized form \eqref{eq:H-unnormalized-02} is the sign-independent
version of the same normal-form calculation and remains valid for
\(s\neq0\).

Hamilton equations with respect to the symplectic form
\(s\,\dd A\wedge \dd B\) read
\begin{equation}
\begin{aligned}
    &A_x=B+\textup{higher-order terms},\\ 
    &B_x
    =
    \frac{\delta}{s}A
    +
    \frac{3C_\ast}{2s}A^2
    +\textup{higher-order terms}.
\end{aligned}  
\end{equation}

Hence, to leading order,
\begin{equation}\label{eq:A-leading-02}
    A_{xx}
    =
    \frac{\delta}{s}A
    +
    \frac{3C_\ast}{2s}A^2.
\end{equation}
The origin is hyperbolic precisely when
\[
    \frac{\delta}{s}>0, \text{ or equivalently } (\alpha-\alpha_\ast)(\beta_0-\beta_\ast)>0.
\]
In this case, the leading-order equation \eqref{eq:A-leading-02} has the
homoclinic solution
\begin{equation}\label{eq:A-homoclinic}
    A(x)
    =
    -\frac{\delta}{C_\ast}
 \operatorname{sech}^2\!\left(
        \frac12\sqrt{\frac{\delta}{s}}\,x
    \right),
\end{equation}
where the square root is taken of the positive number \(\delta/s\). The
reversible homoclinic persistence argument used above then yields a nearby
symmetric homoclinic solution of the full reduced system for
\(0<|\delta|\ll1\) satisfying
\begin{equation}\label{eq:eta-leading-profile}
    \eta(x)
    =
    -\frac{\delta}{C_\ast}
   \operatorname{sech}^2\!\left(
        \frac12\sqrt{\frac{\delta}{s}}\,x
    \right)
    +O(|\delta|^{3/2}).
\end{equation}
Consequently, the leading-order polarity is
\[
    \operatorname{sgn}\eta
    =
    -\operatorname{sgn}\delta
    =
    -\operatorname{sgn}(\alpha-\alpha_\ast).
\]

It remains to express this criterion in terms of the Coriolis parameter.
Recall that
\[
    \alpha=\frac{g_\ast d^3}{m^2}
    =
    \frac{(g-2\Omega c)d^3}{m^2}.
\]
At the Hamiltonian \(0^2\) resonance,
\[
    \alpha_\ast
    =
    \frac{g_\ast^{(0^2)}d^3}{m^2},
    \qquad
    g_\ast^{(0^2)}
    =
    g-2\Omega_c c^{(0^2)},
\]
where
\[
    \Omega_c
    =
    \frac{g-g_\ast^{(0^2)}}{2c^{(0^2)}}.
\]
Therefore, to leading order along the local branch, with
\(c=c^{(0^2)}\) at the bifurcation point,
\[
    \delta
    =
    \alpha-\alpha_\ast
    =
    -\frac{2c^{(0^2)}d^3}{m^2}
    (\Omega-\Omega_c)
    +\textup{higher-order terms}.
\]
Since \(c^{(0^2)}>0\), this implies
\[
    \Omega>\Omega_c
    \quad\Longleftrightarrow\quad
    \delta<0,
\]
and
\[
    \Omega<\Omega_c
    \quad\Longleftrightarrow\quad
    \delta>0.
\]

Combining this with the existence condition
\[
    \frac{\delta}{s}>0,
    \qquad s=\beta_0-\beta_\ast,
\]
we obtain the two possible polarity regimes:
\[
    \beta_0<\beta_\ast
    \quad\text{and}\quad
    \Omega>\Omega_c
    \quad\Longrightarrow\quad
    \delta<0,\ s<0,
\]
and hence \(\eta>0\) to leading order. This gives solitary waves of
elevation. Similarly,
\[
    \beta_0>\beta_\ast
    \quad\text{and}\quad
    \Omega<\Omega_c
    \quad\Longrightarrow\quad
    \delta>0,\ s>0,
\]
and hence \(\eta<0\) to leading order. This gives solitary waves of
depression. The proof of the corollary is complete.
\end{proof}

\subsection{The Hamiltonian real $1{:}1$ resonance}\label{sec:real11}

This section is devoted to the analysis the Hamiltonian real $1{:}1$ resonance of $f$-plane.
The geometric mechanism is the following: two pairs of real spatial eigenvalues
collide at nonzero points on the real axis and then leave the real axis as a complex quartet.
Recall that in the scalar spectral formulation in proposition~\ref{prop:SL-reduction}, each eigenvalue of the scalar problem corresponds to a
double positive eigenvalue $\nu=\kappa^2>0$ of the Sturm--Liouville problem. Hence, positive values of $\nu$ correspond to real spatial
eigenvalues, while negative values of $\nu$ correspond to purely imaginary spatial eigenvalues.

The curve $C_1$ in Figure~\ref{Bifurcation curve} consists of points at which a positive
eigenvalue $\nu=\kappa^2$ has algebraic multiplicity two; equivalently, the line
$s=\hat\alpha-\hat\beta\,\nu$ is tangent to the graph $s=B(\nu)$ at some
$\nu=\kappa_0^2>0$. The local part of $C_1$ near the point
$(\beta_*,\alpha_*)$ is obtained by perturbing away from a codimension-two configuration, $(\beta_*,\alpha_*)$,
for which the operator $\mathcal{L}$ has a quadruple zero eigenvalue. Since the zero eigenvalue has algebraic multiplicity four at the point
$(\beta_*,\alpha_*)$,
the center subspace of the center manifold is given by
\[
X_c=\operatorname{span}\{w_1,w_2,w_3,w_4\},
\]
where $w_i$ are all generalized eigenvectors found in subsection~\ref{quadruple}.

Let \(\Psi\) denote the constant symplectic form associated with the linearized Hamiltonian
system in the flattened variables as used previously in \eqref{02symplectic}. It is not hard to check that, the symplectic pairings satisfy
\begin{equation}\label{eq:04-pairings}
\Psi(w_1,w_4)=-d_1,
\qquad
\Psi(w_2,w_3)=d_1,
\qquad
\Psi(w_3,w_4)=d_2,
\end{equation}
and that all remaining pairings vanish where 
\begin{equation}
    \begin{aligned}
        &d_1=\alpha_\ast \beta_\ast \Theta(1)-\alpha_\ast^2 \int_0^1 a(\zeta) I(\zeta) \Theta(\zeta)\;\dd \zeta>0,\\
        &d_2=\alpha_\ast^2 \beta_\ast \Theta^2(1)-\alpha_\ast^2 \int_0^1 a(\zeta) I(\zeta) \Theta(\zeta)\;\dd \zeta.
    \end{aligned}
\end{equation}

Using both quantities $d_1$ and $d_2$, we define the following normalizations
\begin{equation}\label{eq:04-symplectic-basis}
e_1:=d_1^{-1/2}(w_4+d_3w_2),
\qquad
e_2:=d_1^{-1/2}w_2,
\qquad
f_1:=d_1^{-1/2}w_1,
\qquad
f_2:=d_1^{-1/2}w_3,
\end{equation}
where \(d_3:=d_2/d_1\). 
It is also clear that the linearized operator $\mathcal{L}$ satisfies the following equations
\[
\mathcal{L}e_1=f_2+d_3f_1,\qquad \mathcal{L}e_2=f_1,\qquad \mathcal{L}f_1=0,\qquad \mathcal{L}f_2=e_2,
\]
together with the symplectic relations
\[
\Psi(e_i,e_j)=0,\qquad \Psi(f_i,f_j)=0,\qquad \Psi(e_i,f_j)=\delta_{ij},\qquad \Psi(f_j,e_i)=-\delta_{ij}.
\]

%\subsection{Local Parametrization of $C_1$}
In the present notation, the algebraic multiplicity four of zero translates to the  scalar problem having a double eigenvalue at $\nu=0.$ At this point, the intersection and tangency conditions read
\begin{equation}\label{eq:C1-basepoint}
    B(0)=\widehat{\alpha}_\ast,
    \qquad
    B'(0)=-\widehat{\beta}_*,
\end{equation}
where recall
\begin{equation}\label{eq:C1-tangency-system}
    B(\nu)=\widehat{\alpha}-\widehat{\beta}\nu,
    \qquad
    B'(\nu)=-\widehat{\beta}.
\end{equation}
In \eqref{eq:C1-basepoint}, the first condition determines the critical gravity value, while the second determines the
critical capillarity value.

The crucial observation is that the scalar spectral problem depends on \(\kappa\) only through
\(\nu=\kappa^2\). Consequently, replacing \(\kappa\) by \(-\kappa\) leaves the problem
unchanged. Therefore, the local branch of \(C_1\) is naturally an even function of \(\kappa\),
and so its Taylor expansion contains only even powers of $\kappa$. More precisely, the expansion for $B$ and its derivative read
% \begin{equation}\label{eq:C1-even-expansion}
%     \beta_\kappa=\beta_*+\beta_2\kappa^2+\beta_4\kappa^4+\cdots,
%     \qquad
%     \alpha_\kappa=\alpha_*+\alpha_2\kappa^2+\alpha_4\kappa^4+\cdots,
% \end{equation}
\[
B(\kappa^2)=B(0)+B'(0)\kappa^2+\frac12 B''(0)\kappa^4+O(\kappa^6),
\]
and
\[
B'(\kappa^2)=B'(0)+B''(0)\kappa^2+O(\kappa^4).
\]
Combining this together with  the slope equation in \eqref{eq:C1-tangency-system} yields
\[
\widehat{\beta}-\widehat{\beta}_\ast=-B''(0)\kappa^2+O(\kappa^4).
\]
Similarly, 
\[
\widehat{\alpha}-\widehat{\beta}_\ast\kappa^2- \widehat{\alpha}_\ast=B(\kappa^2)-B(0)=\kappa^2 B'(0)+\frac{\kappa^4}{2}B''(0)+O(\kappa^6).
\]
Thus, using the definition of $\widehat{\alpha}$ and $\widehat{\beta}$, one obtains to the following local parametrization 
\begin{equation}\label{eq:C1-local-expansion}
    \beta=\beta_*+\beta_2\kappa^2+O(\kappa^4),
    \qquad
    \alpha=\alpha_*+\alpha_4\kappa^4+O(\kappa^6).
\end{equation}
% where the subscript $\kappa$ just indicates the dependence of $\alpha$ and $\beta$ on $\kappa$.
Hence, we can conclude that the first correction in \(\beta\) is quadratic, while the first nontrivial
correction in \(\alpha\) is quartic.

In applications, it is convenient to repackage this
parametrization using a small parameter \(\kappa \sim \mu>0\) and a bifurcation parameter
\(\delta\). We introduce
\begin{equation}\label{eq:C1-mu-delta}
    \varepsilon_1=\beta_2(1+\delta)\mu^2,
    \qquad
    \varepsilon_2=\alpha_4\mu^4,
\end{equation}
so that
\begin{equation}\label{eq:C1-parameters}
    \beta=\beta_*+\beta_2(1+\delta)\mu^2,
    \qquad
    \alpha=\alpha_*+\alpha_4\mu^4,
\end{equation}
where \(\beta_2\) and \(\alpha_4\) are the first non-vanishing coefficients. The choice
\(\delta=0\) corresponds to the first-order approximation of \(C_1\), and varying \(\delta\)
moves the parameters to one side or the other of the resonance curve.

\begin{remark}
The role of the codimension-two point is essential. Without the extra degeneracy at
\((\beta_*,\alpha_*)\), one would only conclude that \(\alpha\) and \(\beta\)
are even functions of \(\kappa\). The fact that \(\beta-\beta_*\) begins at order
\(\kappa^2\), while \(\alpha-\alpha_*\) begins at order \(\kappa^4\), is a direct
consequence of the $0^4$ degeneracy of the scalar problem at \(\nu=0\).
\end{remark}

%\subsection{The reduced Hamiltonian and equations}

We now explain in detail how the reduced Hamiltonian here is computed, and how
this leads to the truncated Hamiltonian used in the Hamiltonian real \(1{:}1\) analysis. Let \(\widetilde H^{\varepsilon}\) denote the reduced Hamiltonian on the center manifold,
written in canonical coordinates
\[
\tilde{u}_1=q_1e_1+q_2e_2+p_1f_1+p_2f_2.
\]
Expanding the reduced Hamiltonian near \((\tilde u_1,\varepsilon)=(0,0)\) yields
\begin{equation}\label{eq:H-expansion-real11}
\tilde H^\varepsilon
=
\tilde H^{0,0}_2
+\varepsilon_1\tilde H^{1,0}_2
+\varepsilon_2\tilde H^{0,1}_2
+\varepsilon_1^2\tilde H^{2,0}_2
+\tilde H^{0,0}_3
+\cdots.
\end{equation}
Here,

\begin{itemize}
    \item the subscript denotes the degree in \((q_1,q_2,p_1,p_2)\),
    \item the superscript \((a,b)\) denotes the coefficient of \(\varepsilon_1^a\varepsilon_2^b\).
\end{itemize}
Additionally, due to reversibility, we have the following identity
\[
\tilde H^\varepsilon(-q_1,-q_2,p_1,p_2)=\tilde H^\varepsilon(q_1,q_2,p_1,p_2).
\]
Hence, only monomials that contain an even total number of \(q\)-factors may appear.
Going back to \eqref{eq:H-expansion-real11}, we can express each term as follows,
\begin{equation}\label{eq:expansion}
    \begin{aligned}
        &\tilde{H}^{0,0}_2=c^{0,0}_1 p_1^2+c^{0,0}_2 p_1 p_2+c^{0,0}_3 p_2^2+ c^{0,0}_4q_1^2+c^{0,0}_5q_1q_2 +c^{0,0}_6q_2^2,\\&
        \tilde{H}^{1,0}_2=c^{1,0}_1 p_1^2+c^{1,0}_2 p_1 p_2+c^{1,0}_3 p_2^2+ c^{1,0}_4q_1^2+c^{1,0}_5q_1q_2 +c^{1,0}_6q_2^2,\\&
        \tilde{H}^{0,1}_2=c^{0,1}_1 p_1^2+c^{0,1}_2 p_1 p_2+c^{0,1}_3 p_2^2+ c^{0,1}_4q_1^2+c^{0,1}_5q_1q_2 +c^{0,1}_6q_2^2,\\&
        \tilde{H}^{0,0}_3=c^{0,0}_1 p_1^3+c^{0,0}_2 p_1^2 p_2+c^{0,0}_3 c^{0,0}_4 p_1 p_2^2+ c^{0,0}_5 p_2^3 + c^{0,0}_6 q_1 q_2 p_1 +  c^{0,0}_7 q_1^2p_1+ c^{0,0}_8 q_1^2p_2+c^{0,0}_{9} q_1q_2p_2+ c^{0,0}_{10} q_2^2p_2 .
    \end{aligned}
\end{equation}
Near the local branch of the real \(1{:}1\) curve \(C_1\), we anticipate the following scalings
\begin{equation}\label{scaling}
q_1(x)\sim \mu^7Q_1(X),
\qquad
q_2(x)\sim \mu^5Q_2(X),
\qquad
p_1(x)\sim \mu^4P_1(X),
\qquad
p_2(x)\sim \mu^6P_2(X).
\end{equation}
These scalings assign the weighted degrees to all $q_1,q_2,p_1,p_2, \varepsilon_1,\varepsilon_2$ in the following manner
\[
\deg(q_1)=7,\quad
\deg(q_2)=5,\quad
\deg(p_1)=4,\quad
\deg(p_2)=6,\quad
\deg(\varepsilon_1)=2,\quad
\deg(\varepsilon_2)=4.
\]

We now keep only those monomials of weighted degree at most \(12\), since these are
terms that survive in the leading-order reduced equation. Therefore, in these new scaled variables, some terms will appear as higher order and are absorbed into the
remainder while some others are retained.
Upon anticipation in collecting all terms of weighted degree at most \(12\), we obtain the following truncated reduced hamiltonian
\begin{equation}\label{eq:truncated-H-real11}
\begin{aligned}
\tilde H^\varepsilon(q_1,q_2,p_1,p_2)
&=
c^{0,0}_1p_2^2+c^{0,0}_5 q_1q_2
+c^{1,0}_1\varepsilon_1p_1^2
+c^{1,0}_2\varepsilon_1p_1p_2
+c^{1,0}_6\varepsilon_1q_2^2
+c^{0,1}_1\varepsilon_2p_1^2
+c^{2,0}_1\varepsilon_1^2p_1^2
+c^{0,0}_1\,p_1^3\\&\qquad
+R(q_1,q_2,p_1,p_2),
\end{aligned}
\end{equation}
where $R$ denotes the remainder portion.
 Next, we will compute each coefficient in \eqref{eq:truncated-H-real11} explicitly. We begin with

\paragraph{\underline{\textbf{ Computation for \(c^{0,0}_1\) and \(c^{0,0}_5\)}}}
Notice, first of all, that
\[
\tilde H^{0,0}_2(q_1,q_2,p_1,p_2)=K^{0,0}_2[\tilde u_1,\tilde u_1].
\]
We compute the relevant bilinear terms.
First, 
\[
c^{0,0}_5=2K^{0,0}_2[e_1,e_2]
=
\Psi(\mathcal{L}e_1,e_2)
=
\Psi(f_2+d_3f_1,e_2)=-1.
\]
Further, observe that
\[
2K^{0,0}_2[f_2,f_2]
=
\Psi(\mathcal{L}f_2,f_2)
=
\Psi(e_2,f_2)=1,
\]
so
\[
c^{0,0}_1=K^{0,0}_2[f_2,f_2]=\frac12.
\]

\paragraph{\underline{\textbf{ Computation for \(c^{1,0}_1\), \(c^{1,0}_2\) and \(c^{1,0}_6\)}}}
We begin by computing $c^{1,0}_1$ and $c^{1,0}_6$. Let $u=\tilde{u_1}+\tilde{r}(\tilde{u}_1;\varepsilon)$ lies on the center manifold. Then, one can check that we have the following identity
\begin{equation}\label{identity}
\mathcal{L}\tilde r(\tilde u_1;\varepsilon)
-
d_1\tilde r[\tilde u_1;\varepsilon](\mathcal{L}\tilde u_1)
=
-
\mathcal{N}^\varepsilon\!\bigl(\tilde u_1+\tilde r(\tilde u_1;\varepsilon)\bigr)
+
d_1\tilde r[\tilde u_1;\varepsilon]\bigl(P^\varepsilon(\tilde u_1)\bigr)
+
P^\varepsilon(\tilde u_1),
\end{equation} where $P^\varepsilon(\tilde u_1):=\mathcal{L}\tilde{u}_1 +\mathcal{PN}(\tilde{u}_1+r(\tilde{u}_1;\varepsilon);\varepsilon)$ is the reduced hamiltonian equation. It is not hard to see that
\[
P^{1,0}_1(\tilde u_1)= \frac{\partial \tilde{H}^{1,0}_2}{\partial p_1} e_1 +\frac{\partial \tilde{H}^{1,0}_2}{\partial p_2} e_2 -\frac{\partial \tilde{H}^{1,0}_2}{\partial q_1} f_1-\frac{\partial \tilde{H}^{1,0}_2}{\partial q_2} f_2.
\]
This combined together with the expression of $\tilde{H}^{1,0}_2$ in \eqref{eq:expansion} allows us to infer
\begin{equation}\label{eq:P110}
P^{1,0}_1(\tilde u_1)
=
2c^{1,0}_1p_1e_1
+
c^{1,0}_2p_2e_1
+
c^{1,0}_2p_1e_2
+
2c^{1,0}_3p_2e_2
-
(2c^{1,0}_4q_1+c^{1,0}_5q_2)f_1
-
(c^{1,0}_5q_1+2c^{1,0}_6q_2)f_2.
\end{equation}

Moreover, straightforward but tedious calculations yield 
\begin{equation}\label{eq:N110-values}
\mathcal{N}^{1,0}_1(e_1)=0,
\qquad
\mathcal{N}^{1,0}_1(e_2)=0,
\qquad
\mathcal{N}^{1,0}_1(f_1)=0,
\qquad
\mathcal{N}^{1,0}_1(f_2)=\left(0,d_1^{-1/2}\beta_*^{-1}\alpha_*^{-1}a(1),0,0\right).
\end{equation}

Using \eqref{identity} together with the expression of $P^{1,0}_1(\tilde u_1)$ and comparing the coefficients of \(\varepsilon_1\tilde u_1\),  we obtain
the system
\begin{equation}\label{eq:homo-1000}
\mathcal{L}\tilde r^{1,0}_{1000}
-
\tilde r^{1,0}_{0001}
-
d_3\tilde r^{1,0}_{0010}
=
-
\mathcal{N}^{1,0}_1[e_1]
-
2c^{1,0}_4f_1
-
c^{1,0}_5f_2,
\end{equation}
\begin{equation}\label{eq:homo-0100}
\mathcal{L}\tilde r^{1,0}_{0100}
-
\tilde r^{1,0}_{0010}
=
-
\mathcal{N}^{1,0}_1[e_2]
-
c^{1,0}_5f_1
-
2c^{1,0}_6f_2,
\end{equation}
\begin{equation}\label{eq:homo-0010}
\mathcal{L}\tilde r^{1,0}_{0010}
=
-
\mathcal{N}^{1,0}_1[f_1]
+
2c^{1,0}_1e_1,
\end{equation}
\begin{equation}\label{eq:homo-0001}
\mathcal{L}\tilde r^{1,0}_{0001}
-
\tilde r^{1,0}_{0100}
=
-
\mathcal{N}^{1,0}_1[f_2]
+
2c^{1,0}_3e_2.
\end{equation}
Using \eqref{eq:N110-values}, these simplify to
\begin{equation}\label{eq:homo-1000-simplified}
\mathcal{L}\tilde r^{1,0}_{1000}
-
\tilde r^{1,0}_{0001}
-
d_3\tilde r^{1,0}_{0010}
=
-
2c^{1,0}_4f_1
-
c^{1,0}_5f_2,
\end{equation}
\begin{equation}\label{eq:homo-0100-simplified}
\mathcal{L}\tilde r^{1,0}_{0100}
-
\tilde r^{1,0}_{0010}
=
-
c^{1,0}_5f_1
-
2c^{1,0}_6f_2,
\end{equation}
\begin{equation}\label{eq:homo-0010-simplified}
\mathcal{L}\tilde r^{1,0}_{0010}
=
2c^{1,0}_1e_1,
\end{equation}
\begin{equation}\label{eq:homo-0001-simplified}
\mathcal{L}\tilde r^{1,0}_{0001}
-
\tilde r^{1,0}_{0100}
=
-
\mathcal{N}^{1,0}_1[f_2]
+
2c^{1,0}_3e_2.
\end{equation}
Using \eqref{eq:homo-0010-simplified}, we obtain
\[
\mathcal{L}\tilde r^{1,0}_{0010}=2c^{1,0}_1e_1.
\]
We pair both sides with \(f_1\) via \(\Psi\) and obtain
\[
\Psi(\mathcal{L}\tilde r^{1,0}_{0010},f_1)
=
2c^{1,0}_1\Psi(e_1,f_1)
=
2c^{1,0}_1.
\]
On the other hand, since \(\mathcal{L}f_1=0\), the skew-symmetry property of the Hamiltonian linearization leads to the following equation
\[
\Psi(\mathcal{L}\tilde r^{1,0}_{0010},f_1)
=
-\Psi(\tilde r^{1,0}_{0010},\mathcal{L}f_1)
=
0.
\]
Therefore, we can conclude that $c^{1,0}_1=0.$
Substituting this back into \eqref{eq:homo-0010-simplified} gives $
\mathcal{L}\tilde r^{1,0}_{0010}=0.$
Since the kernel of \(\mathcal{L}\) on the center subspace is spanned by \(f_1\), it follows that
\begin{equation}\label{eq:r0010-form}
\tilde r^{1,0}_{0010}=c f_1
\end{equation}
for some scalar \(c\).
Now, we use \eqref{eq:homo-0100-simplified} together with \eqref{eq:r0010-form} to deduce
\[
\mathcal{L}\tilde r^{1,0}_{0100}-cf_1
=
-
c^{1,0}_5f_1
-
2c^{1,0}_6f_2.
\]
Rearranging the terms yields
\begin{equation}\label{eq:r0100-eqn}
\mathcal{L}\tilde r^{1,0}_{0100}
=
(c-c^{1,0}_5)f_1
-
2c^{1,0}_6f_2.
\end{equation}

As before, we now pair \eqref{eq:r0100-eqn} with \(e_2\). On the right-hand side we get
\[
\Psi\big((c-c^{1,0}_5)f_1-2c^{1,0}_6f_2,e_2\big)
=
(c-c^{1,0}_5)\Psi(f_1,e_2)-2c^{1,0}_6\Psi(f_2,e_2).
\]
Hence,
\[
\Psi\big((c-c^{1,0}_5)f_1-2c^{1,0}_6f_2,e_2\big)=2c^{1,0}_6.
\]
On the left-hand side, we have
\[
\Psi(\mathcal{L}\tilde r^{1,0}_{0100},e_2)
=
-\Psi(\tilde r^{1,0}_{0100},\mathcal{L}e_2)
=
-\Psi(\tilde r^{1,0}_{0100},f_1),
\]
where we have used \(\mathcal{L}e_2=f_1\). Hence, we obtain
\begin{equation}\label{eq:c160-intermediate}
2c^{1,0}_6=-\Psi(\tilde r^{1,0}_{0100},f_1).
\end{equation}

Next, we use \eqref{eq:homo-0001-simplified}:
\[
\mathcal{L}\tilde r^{1,0}_{0001}-\tilde r^{1,0}_{0100}
=
-
\mathcal{N}^{1,0}_1[f_2]
+
2c^{1,0}_3e_2.
\]
Pairing this with \(f_1\) and using the fact that \(\mathcal{L}f_1=0\), the first term on the left gives
\[
\Psi(\mathcal{L}\tilde r^{1,0}_{0001},f_1)=0.
\]
Also, since \(\Psi(e_2,f_1)=0\), therefore, we have
\[
-\Psi(\tilde r^{1,0}_{0100},f_1)
=
-\Psi(\mathcal{N}^{1,0}_1[f_2],f_1).
\]
This is equivalent to saying
\begin{equation}\label{eq:r0100-pairing}
\Psi(\tilde r^{1,0}_{0100},f_1)=\Psi(\mathcal{N}^{1,0}_1[f_2],f_1).
\end{equation}
Combining \eqref{eq:c160-intermediate} and \eqref{eq:r0100-pairing}, we obtain
\begin{equation}\label{eq:c160-from-N}
2c^{1,0}_6=-\Psi(\mathcal{N}^{1,0}_1[f_2],f_1).
\end{equation}
Thus, it remains to evaluate this symplectic pairing. Using \eqref{eq:N110-values}, we have
\[
\mathcal{N}^{1,0}_1(f_2)=\left(0,d_1^{-1/2}\beta_*^{-1}\alpha_*^{-1}a(1),0,0\right).
\]
Since \(\Psi\) pairs the second component of the first vector with the first component of the
second vector, it follows that
\[
\Psi(\mathcal{N}^{1,0}_1[f_2],f_1)
=
\frac{1}{d_1\alpha_*^2}.
\]
Substituting this into \eqref{eq:c160-from-N} yields
\[
c^{1,0}_6=-\frac{1}{2d_1\alpha_*^2}=-\frac{1}{\beta_2}.
\]
The coefficient $c^{1,0}_2$ vanishes and can be shown using Proposition 4.1 in \cite{GrovesWahlen2007}.
\paragraph{\underline{\textbf{Computation for \(c^{0,1}_1\)}}}
The coefficient $c^{0,1}_1$ can be computed as follows
\[
c^{0,1}_1= K^{0,1}_2[f_1,f_1]+2K^{0,0}_2[f_1,r^{0,1}_{0010}]=-\frac{1}{2d_1\alpha_*^2}=-\frac{1}{\beta_2}.
\]
\paragraph{\underline{\textbf{ Computation for \(c^{2,0}_1\)}}}
The same reasoning as above implies $c^{2,0}_1=0$.
\\
\paragraph{\underline{\textbf{ Computation for \(c^{0,0}_1\)}}}
Finally, the coefficient of the cubic term $p_1^3$ is
\[
c^{0,0}_1=\frac{1}{3}\Psi(\mathcal{N}^{0,0}_2[f_1,f_1],f_1)=\frac{-1}{2d^{3/2}_1}\int_0^1 a(s)^{-5}\,\dd s.
\]

Recalling the scaling in \eqref{scaling} and the coefficients we computed above, we obtain
\begin{equation}
\tilde H^\varepsilon(Q_1,Q_2,P_1,P_2)
=
\mu^{12} \left(-\frac{1}{2}P_1^2- Q_1Q_2
-(1+\delta)Q_2^2
+c^{0,0}_1\,P_1^3\right)
+O(\mu^{14}).
\end{equation}
As a result, we derive the corresponding Hamilton equations
\begin{equation}
    \begin{aligned}
        &Q_{1X}=-P_1+3c^{0,0}_1P_1^2+O(\mu)\\&
        Q_{2X}=P_2+O(\mu)\\&
        P_{1X}=Q_2+O(\mu)\\&
        P_{2X}=2(1+\delta)Q_2+Q_1+O(\mu)
    \end{aligned}
\end{equation}
Upon taking $\mu\to 0$, we can rewrite the above system as a single equation in terms of $u:=3c^{0,0}_1P_1$ and its derivatives in which we get

\begin{equation}\label{eq:real11-model}
u_{XXXX}-2(1+\delta)u_{XX}+u-u^2
=
O(\mu).
\end{equation}
The essential point is that the
four-dimensional center-manifold collapses, after normalization, to a single
reversible fourth-order differential equation.

The phase portrait of \eqref{eq:real11-model} is well known: it possesses a primary reversible homoclinic
orbit and, by standard results of Devaney \cite{devaney1976homoclinic} type, an infinite family of multipulse homoclinic
solutions. These correspond to solitary waves whose profiles have several large troughs,
separated by small oscillations, and which approach the laminar state with an oscillatory
exponentially decaying tail.

\subsection{The Hamiltonian--Hopf bifurcation}\label{sec:hopf}

We next formulate the \(f\)-plane analogue of the Hamiltonian--Hopf bifurcation
treated by Groves and Wahl\'en at points of the curve \(C_2\).
In spectral terms, this occurs when two pairs of purely imaginary spatial eigenvalues
collide at nonzero points on the imaginary axis and then leave the axis as a
complex quartet. In the scalar Sturm--Liouville formulation, this collision point on the imaginary axis corresponds to a
double negative eigenvalue
\[
\nu=-\mathfrak{q}^2<0
\]
of the scalar problem \eqref{eq:SL-problem}.

A tangent intersection at a negative value \(\nu=-\mathfrak{q}^2<0\) corresponds to an
algebraically double pair of imaginary eigenvalues \(\kappa=\pm i\mathfrak{q}\).
The set of parameter values at which this happens defines the analogue of
Groves--Wahl\'en curve \(C_2\). We therefore fix a reference point
\[
(\beta_0,\alpha_0)\in C_2
\]
and choose the perturbations
\begin{equation}\label{eq:hopf-parameters}
(\varepsilon_1,\varepsilon_2)=(0,\delta),
\qquad
0<\delta\ll1.
\end{equation}
Thus, \(\beta=\beta_0\) is held fixed, while \(\alpha=\alpha_0+\delta\) varies transversely
to the critical curve. This is exactly the parameter choice used by Groves and Wahl\'en  \cite{GrovesWahlen2007}
for the Hamiltonian--Hopf bifurcation. At \(\delta=0\), the Hamiltonian--Hopf curve is characterized by a tangency of the scalar
spectral graph at a negative value \(\nu=-\mathfrak{q}^2<0\), so that the spatial eigenvalues are
\(\pm i\mathfrak{q}\) with algebraic multiplicity two. For \(|\delta|\ll1\), this tangency unfolds:
on one side, \(\delta<0\), the double negative eigenvalue splits into two distinct negative
eigenvalues, corresponding to two purely imaginary pairs of spatial eigenvalues; on the other
side, \(\delta>0\), the negative scalar eigenvalue ceases to be real, and the corresponding spatial eigenvalues
leave the imaginary axis to form a complex quartet.

Assume that the linearized operator \(\mathcal{L}\) at \((\beta_0,\alpha_0)\) has eigenvalues
\(\pm i\mathfrak{q}\) of algebraic multiplicity two, with generalized eigenvectors \(e,f,\bar e,\bar f\)
satisfying
\begin{equation}\label{eq:hopf-jordan}
\mathcal{L}e=i\mathfrak{q}\,e,
\qquad
\mathcal{L}\bar e=-i\mathfrak{q}\,\bar e,
\qquad
(\mathcal{L}-i\mathfrak{q}\mathcal{I})f=e,
\qquad
(\mathcal{L}+i\mathfrak{q}\mathcal{I})\bar f=\bar e.
\end{equation}
We would like to note that the point $(\beta_0,\alpha_0)$ exists locally near the point $(\beta_\ast,\alpha_\ast)$ from the Hamiltonian 1:1 resonance. 
Modifying \(f\) by a suitable multiple of \(e\), if necessary, we may suppose that
\begin{equation}\label{eq:hopf-symplectic}
\Psi(e,\bar f)=1,
\qquad
\Psi(f,\bar e)=-1,
\end{equation}
and that all other symplectic products of these vectors vanish.
Furthermore, the reversibility condition may be arranged so that
\begin{equation}\label{eq:hopf-reverser}
Se=\bar e,
\qquad
Sf=-\bar f.
\end{equation}
It follows that \(\{e,f,\bar e,\bar f\}\) is a symplectic basis for the
four-dimensional center subspace
\[
X_c=\operatorname{span}\{e,f,\bar e,\bar f\}.
\] 
We also introduce complex canonical coordinates \((A,B)\in\C^2\) by writing
\begin{equation}\label{eq:hopf-coordinates}
\tilde{u}_1=Ae+Bf+\bar A\,\bar e+\bar B\,\bar f.
\end{equation}
In these coordinates the reverser acts by
\[
S(A,B)=(\bar A,-\bar B).
\]

%\subsection{The reduced Hamiltonian and Birkhoff normal form}

Let \(\widetilde H^{0,\delta}(A,B)\) denote the reduced Hamiltonian on the center manifold. The flow of the Hamiltonian can be analyzed using the theory developed by Iooss \& P\'erou\`eme \cite{IoossPeroueme1993} and Buffoni \& Groves \cite{BuffoniGroves1999}.
By Birkhoff normal-form theory for reversible Hamiltonian systems, for each \(n_0\ge2\),
there is a near-identity analytic symplectic change of coordinates which transforms the
reduced Hamiltonian into the form
\begin{equation}\label{eq:hopf-normal-form}
\widetilde H^{0,\delta}(A,B)
=
iq(A\bar B-\bar A B)+|B|^2
+
H_{\mathrm{NF}}\!\left(|A|^2,\ i(A\bar B-\bar A B),\ \delta\right)
+
O\!\left(|(A,B)|^2 |(\delta,A,B)|^{n_0}\right),
\end{equation}
where \(H_{\mathrm{NF}}\) is a real polynomial of order $n_0+1$ in the invariants
\[
I_1:=|A|^2,
\qquad
I_2:=i(A\bar B-\bar A B),
\qquad
I_3:=\delta.
\]
The associated Hamilton equations are
\begin{equation}\label{eq:hopf-eq-A}
A_x
=
i\mathfrak{q}A+B
+
iA\,\partial_2H_{\mathrm{NF}}
+
O\!\left(|(A,B)|\,|(\delta,A,B)|^{n_0}\right),
\end{equation}
\begin{equation}\label{eq:hopf-eq-B}
B_x
=
i\mathfrak{q}B
+
iB\,\partial_2H_{\mathrm{NF}}
-
A\,\partial_1H_{\mathrm{NF}}
+
O\!\left(|(A,B)|\,|(\delta,A,B)|^{n_0}\right).
\end{equation}
Since $H_{NF}(0,0,\delta)=0$ and the origin is an equilibrium, the lowest-order terms are
obtained by writing the most general real polynomial in $I_1$, $I_2$, and $\delta$ of low order, namely
\[
H_{NF}
=
\delta c_1 I_1+\delta c_2 I_2+c_3 I_1^2+c_4 I_1I_2-c_5 I_2^2+\delta^2 c_6 I_1+\delta^2 c_7 I_2+\cdots.
\]
The theory in \cite{IoossPeroueme1993} and \cite{BuffoniGroves1999} that we will use below, requires that $c_1<0$ and $c_3>0$ from the above expression of $H_{NF}$.
Their theory can be summarized as follows.

\begin{theorem}[Envelope solitary waves from the Hamiltonian--Hopf bifurcation]\label{thm:HH-fplane}

Suppose that
\[
c_1<0,
\qquad
c_3>0.
\]
Then the following statements hold.

\begin{enumerate}[label=\textup{(\roman*)}]
    \item \cite{IoossPeroueme1993} For each sufficiently small positive value of \(\delta\), the reduced two-degree-of-freedom Hamiltonian system associated with the truncated normal form possesses two distinct symmetric homoclinic solutions.

    \item \cite{BuffoniGroves1999} For each sufficiently small positive value of \(\delta\), the same reduced Hamiltonian system possesses infinitely many geometrically distinct homoclinic solutions which, generically, resemble multiple copies of one of the primary homoclinic solutions in \textup{(i)}.
\end{enumerate}

These homoclinic solutions correspond, via the center-manifold reduction and the inverse of the flattening and hodograph changes of variables, to small-amplitude envelope solitary-wave solutions of the two-dimensional steady \(f\)-plane capillary-gravity problem with general vorticity. Their amplitude is of order
\[
O\bigl(( -c_1\delta)^{1/2}\bigr),
\]
and they decay exponentially to the underlying laminar horizontal flow as
\[
x\to \pm\infty.
\]
\end{theorem}

To be able to use the above theorem, it remains to show that the sign conditions \(c_1<0\) and \(c_3>0\) are satisfied. Since the expressions that we will obtain for both constants are not explicit, we can only guarantee that the  sign conditions hold for $(\beta_0,\alpha_0)$ near the point $(\beta_\ast,\alpha_\ast)$ locally on the branch of \(C_2\). To that end, we introduce a small parameter \(\mu>0\) and write
\[
(\beta_0,\alpha_0)=(\beta_\mu,\alpha_\mu),
\]
where
\[
\beta_\mu=\beta_* - 2\alpha_*^2 d_1\,\mu^2 + O(\mu^4),
\qquad
\alpha_\mu=\alpha_*+\alpha_*^2 d_1\,\mu^4+O(\mu^6),
\]
exactly taken from \eqref{eq:C1-local-expansion}.
\begin{proposition}[Hamiltonian--Hopf normal-form signs]\label{prop:HH-signs}
Let \((\beta_\mu,\alpha_\mu)\) be the local Hamiltonian--Hopf branch near
\((\beta_*,\alpha_*)\), parametrized by \(0<\mu\ll1\). Then the normal-form
coefficients \(c_1\) and \(c_3\) satisfy
\[
c_1
=
-\frac{1}{4d_1}\alpha_*^{-2}\mu^{-2}
+O(1),
\]
and
\[
c_3
=
\frac{19}{64d_1^3}
\left(\int_0^1a(\zeta)^{-5}\,\dd\zeta\right)^2
\mu^{-8}
+
O(\mu^{-6}).
\]
Consequently,
\[
    c_1<0<c_3
\]
for all sufficiently small \(\mu>0\).
\end{proposition}

\begin{proof}
We let
\[
\mathfrak{q}=\mu.
\]
Then one constructs vectors \(e_0\) and \(f_0\) satisfying
\[
\cL e_0=i\mu\,e_0,
\qquad
(\cL-i\mu \mathcal{I})f_0=e_0.
\]
Their components admit asymptotic expansions of the form
\[
e_0=
\begin{pmatrix}
i\mu\alpha_*^{-1}
+i\mu^3 \Theta(1)
+O_i(\mu^5)
\\[1.5ex]
\mu^2 a(1)\alpha_*^{-1}
+\mu^4 a(1)\Theta(1)
+O_r(\mu^6)
\\[1.5ex]
i\mu I(\zeta)
+i\mu^3 \Theta(\zeta)
+O_i(\mu^5)
\\[1.5ex]
\mu^2 a(\zeta)I(\zeta)
+\mu^4 a(\zeta)\Theta(\zeta)
+O_r(\mu^6)
\end{pmatrix},
\]
and
\[
f_0=
\begin{pmatrix}
2\mu^2 \Theta(1)
+O_r(\mu^4)
\\[1.5ex]
-i\mu a(1)\alpha_*^{-1}
-3i\mu^3 a(1)\Theta(1)
+O_i(\mu^5)
\\[1.5ex]
2\mu^2 \Theta(\zeta)
+O_r(\mu^4)
\\[1.5ex]
-i\mu a(\zeta)I(\zeta)
-3i\mu^3 a(\zeta)\Theta(\zeta)
+O_i(\mu^5)
\end{pmatrix}.
\]
Here, \(O_r(\mu^n)\) and \(O_i(\mu^n)\) denote, respectively, real-valued and purely
imaginary quantities of order \(O(\mu^n)\). Additionally, 
\[
\Theta(\zeta):=\int_0^\zeta a(r)^{-3}\int_0^r a(t)I(t)\,\dd t\,\dd r.
\]

These expansions are the \(f\)-plane analogue of the Groves--Wahl\'en expansions near the
codimension-two point \((\beta_*,\alpha_*)\), and are obtained by expanding the
eigenvectors at the Hamiltonian--Hopf point in terms of the Jordan chain
\[
\cL w_1=0,\qquad
\cL w_2=w_1,\qquad
\cL w_3=w_2,\qquad
\cL w_4=w_3.
\]
In particular, one has 
\[
e_0=i\mu\alpha_*^{-1}w_1-\mu^2\alpha_*^{-1}w_2-i\mu^3\alpha_*^{-1}w_3+\mu^4\alpha_*^{-1}w_4+O(\mu^5),
\]
\[
f_0=i\mu\alpha_*^{-1}w_2-2\mu^2\alpha_*^{-1}w_3-3i\mu^3\alpha_*^{-1}w_4+O(\mu^4),
\]
and the above componentwise formulas follow by substituting the explicit \(f\)-plane expressions
for \(w_1,w_2,w_3,w_4\).

The next step is to normalize the generalized eigenvectors. Via direct calculations, one finds
\[
\Psi(e_0,f_0)=d_4,
\qquad
\Psi(f_0,\bar f_0)=i\,d_5,
\]
with
\[
d_4=4d_1\mu^4+O_r(\mu^6),
\qquad
d_5=-4d_1\mu^3+O_r(\mu^5).
\]
Moreover, the identities
\[
\Psi(e_0,\bar e_0)=0,
\qquad
\Psi(e_0,\bar f_0)=0,
\qquad
\Psi(\bar e_0,\bar f_0)=0
\]
follow from the symplectic relations and the eigenvalue equations above. Hence, one may define
\[
e=i\,d_4^{-1/2}e_0,
\qquad
f=i\,d_4^{-1/2}\left(f_0-\frac{i\,d_5}{2d_4}e_0\right),
\]
so that \(e,f,\bar e,\bar f\) satisfy the required symplectic normalization.

The coefficients \(c_1\) and \(c_3\) in the Hamiltonian--Hopf normal form are then given by
\[
c_1=\frac{1}{d_4}\,\Psi\!\bigl(\mathcal{N}^{0,1}_1[e_0],\bar e_0\bigr),
\]
and
\[
c_3=
\frac{1}{d_4^2}\Psi\!\bigl(\mathcal{N}^{0,0}_2[e_0,\tilde R^{0,0}_{1010}],\bar e_0\bigr)
+
\frac{1}{d_4^2}\Psi\!\bigl(\mathcal{N}^{0,0}_2[\bar e_0,\tilde R^{0,0}_{2000}],\bar e_0\bigr)
+
\frac{3}{2d_4^2}\Psi\!\bigl(\mathcal{N}^{0,0}_3[e_0,e_0,\bar e_0],\bar e_0\bigr).
\]

Using the expansions above and the homological equations for
\(\widetilde R^{0,0}_{1010}\) and \(\widetilde R^{0,0}_{2000}\), we obtain
the following leading-order asymptotics for the Hamiltonian--Hopf
normal-form coefficients:
\[
c_1
=
-\frac{1}{4d_1}\alpha_*^{-2}\mu^{-2}
+O_r(1),
\]
so in particular $c_1<0$ for sufficiently small \(\mu>0\).

Similarly, the cubic coefficient takes the form
\[
c_3
=
\frac{19}{64d_1^3}
\left(\int_0^1 a(s)^{-5}\,\dd s\right)^2
\mu^{-8}
+O_r(\mu^{-6}),
\]
and hence $c_3>0$ for sufficiently small \(\mu>0\).

The auxiliary vectors appearing above solve the homological equations
\[
\cL\tilde R^{0,0}_{1010}=-2\mathcal{N}^{0,0}_2[e_0,\bar e_0],
\qquad
(\cL-2i\mu I)\tilde R^{0,0}_{2000}=-\mathcal{N}^{0,0}_2[e_0,e_0].
\]
Their leading-order asymptotics are given by
\[
\tilde R^{0,0}_{1010}=
\begin{pmatrix}
-3d_1^{-1}\alpha_*^{-1}\mu^{-2}\displaystyle\int_0^1 a(s)^{-5}\,\dd s+O_r(1)
\\[1.5ex]
0
\\[1.5ex]
-3d_1^{-1}\mu^{-2}\displaystyle\int_0^1 a(s)^{-5}\,\dd s
\displaystyle\int_0^\zeta a(t)^{-3}\,\dd t
+O_r(1)
\\[1.5ex]
0
\end{pmatrix},
\]
and
\[
\tilde R^{0,0}_{2000}=
\begin{pmatrix}
\frac16 d_1^{-1}\alpha_*^{-1}\mu^{-2}\displaystyle\int_0^1 a(s)^{-5}\,\dd s+O_r(1)
\\[1.5ex]
-\frac13 i\,d_1^{-1}\alpha_*^{-1}\mu^{-1}a(1)\displaystyle\int_0^1 a(s)^{-5}\,\dd s+O_i(\mu)
\\[1.5ex]
\frac16 d_1^{-1}\mu^{-2}\displaystyle\int_0^1 a(s)^{-5}\,\dd s
\displaystyle\int_0^\zeta a(t)^{-3}\,\dd t
+O_r(\mu)
\\[1.5ex]
-\frac13 i\,d_1^{-1}\mu^{-1}a(\zeta)\displaystyle\int_0^1 a(s)^{-5}\,\dd s
\displaystyle\int_0^\zeta a(t)^{-3}\,\dd t
+O_i(\mu)
\end{pmatrix}.
\]

Consequently, the sign hypotheses required in the Hamiltonian--Hopf theorem
are satisfied for all points \((\beta_\mu,\alpha_\mu)\) on the local branch
\(C_2\) sufficiently near \((\beta_*,\alpha_*)\).
\end{proof}
%%%%%%%%%%%%%%%%%%%%%%%%%%%%%%%%%%%%%%%

\section{The weak-effective-gravity regime and its interaction with the low spectrum}\label{sec:weak-gravity}

This section isolates the $f$-plane-specific spectral consequences of the weak-effective-gravity
regime
\[
\alpha \approx 0,
\qquad\text{equivalently}\qquad
g_*=g-2\Omega c \approx 0.
\]
Our aim is to prove both Theorem~\ref{thm:nearstag} and Theorem~ \ref{thm:HH}.  We show that, for laminar flows uniformly
away from stagnation, the weak-effective-gravity regime is spectrally separated from the local
\(0^2\) and Hamiltonian--Hopf bifurcation curves.

% \subsection{The first eigenvalue branch near the \(0^2\) threshold}\label{subsec:first-eigenvalue}

% Here one studies the simple eigenvalue branch \(\nu_0(\alpha)\) passing through \(\nu=0\) at the
% critical value \(\alpha=\alpha_*\), and derives the local expansion
% \[
% \nu_0'(\alpha_*)=\frac{1}{\beta-\beta_*},
% \qquad
% \nu_0(\alpha)=\frac{\alpha-\alpha_*}{\beta-\beta_*}+O((\alpha-\alpha_*)^2).
% \]
% This determines the sign of the first eigenvalue on either side of the \(0^2\) threshold.

\subsection{Separation of the $0^2$ resonance from weak-effective gravity}
\label{sec:02-weak-gravity-proof}

Here, we prove Theorem~\ref{thm:nearstag}. The key point is that the
Hamiltonian $0^2$ resonance occurs at a critical value
\[
    \alpha=\alpha_*,
\]
and this critical value is determined explicitly by the underlying laminar
flow.

\begin{proof}[Proof of Theorem~\ref{thm:nearstag}]
From the zero-eigenvalue calculation in Section~\ref{sec:CMR}, the
Hamiltonian \(0^2\) resonance occurs at
\[
    \alpha=\alpha_*,
    \qquad
    \alpha_*=
    \left(
        \int_0^1 a(\zeta)^{-3}\,\dd\zeta
    \right)^{-1},
\]
provided \(\beta\neq\beta_*\). Thus the question of whether the \(0^2\)
resonance can occur near the weak-effective-gravity threshold \(\alpha=0\)
reduces to whether \(\alpha_*\) can approach zero.

Suppose first that the laminar flows remain uniformly separated from
stagnation, so that
\[
    a(\zeta)\geq a_{\min}>0,
    \qquad 0\leq \zeta\leq 1,
\]
uniformly along the family. Since the nondimensional laminar flow satisfies
\[
    \int_0^1 a(\zeta)^{-1}\,\dd\zeta=H(1)=1,
\]
we have
\[
    a(\zeta)^{-3}
    =
    a(\zeta)^{-2}a(\zeta)^{-1}
    \leq
    a_{\min}^{-2}a(\zeta)^{-1}.
\]
Therefore
\[
    \int_0^1 a(\zeta)^{-3}\,\dd\zeta
    \leq
    a_{\min}^{-2},
\]
and hence
\[
    \alpha_*
    =
    \left(
        \int_0^1 a(\zeta)^{-3}\,\dd\zeta
    \right)^{-1}
    \geq
    a_{\min}^2>0.
\]
Thus \(\alpha_*\), and therefore the corresponding critical effective gravity,
is bounded away from zero for any uniformly non-stagnant family of laminar
flows.

In dimensional variables,
\[
    g_*^{(0^2)}
    =
    \frac{m^2}{d^3}\alpha_*.
\]
Consequently, \(g_*^{(0^2)}\to0\) can occur only if no uniform positive lower
bound for \(a\) is available, that is, only if
\[
    \min_{\zeta\in[0,1]}a(\zeta)\to0.
\]
Finally, for a laminar flow,
\[
    h_\zeta=H_\zeta=\frac1a,
    \qquad
    h_\zeta=\frac1{\psi_z},
\]
so \(a=\psi_z\). Since \(\psi_z=u-c\), up to the fixed nondimensional scaling,
\[
    \min_{\zeta\in[0,1]}a(\zeta)\to0
    \qquad\Longleftrightarrow\qquad
    \inf_D(u-c)\to0.
\]
Hence, the weak-effective-gravity limit of the \(0^2\) resonance can be
reached only in the near-stagnation regime.
\end{proof}

\subsection{Persistence and separation from weak-effective gravity}
\label{sec:HH-weak-gravity-proof}

We now prove Theorem~\ref{thm:HH}. The proof is based on the tangency
conditions for the scalar spectral problem. We use the original
\(\zeta\)-variable formulation, since in this form the effective-gravity
parameter \(\alpha\) and the capillarity parameter \(\beta\) enter the
boundary condition directly.

\begin{proof}[Proof of Theorem~\ref{thm:HH}]
Consider the scalar spectral problem
\begin{equation}\label{eq:zeta-spectral-HH-proof}
-\,a(\zeta)^{-1}\bigl(a(\zeta)^3\phi_\zeta\bigr)_\zeta
=
\nu\phi,
\qquad 0<\zeta<1,
\end{equation}
with boundary conditions
\begin{equation}\label{eq:zeta-bc-HH-proof}
\phi(0)=0,
\qquad
-\frac{a(1)^3}{\beta}\phi_\zeta(1)
+\frac{\alpha}{\beta}\phi(1)
=
\nu\phi(1).
\end{equation}
For each \(\nu\) near zero, let \(M(\cdot;\nu)\) be the solution of
\begin{equation}\label{eq:M-IVP-HH-proof}
\bigl(a^3M_\zeta\bigr)_\zeta=-\nu aM,
\qquad
M(0;\nu)=0,
\qquad
M_\zeta(0;\nu)=a(0)^{-3}.
\end{equation}
The normalization of \(M\) is immaterial for the quotient below. We choose
\(M_\zeta(0;\nu)=a(0)^{-3}\) only so that the solution at \(\nu=0\) is
\begin{equation}\label{eq:I-def-HH-proof}
    I(\zeta):=\int_0^\zeta a(s)^{-3}\,\dd s.
\end{equation}
We define
\begin{equation}\label{eq:B-zeta-proof}
    \mathfrak{B}(\nu):=
    \frac{a(1)^3M_\zeta(1;\nu)}{M(1;\nu)}.
\end{equation}
Since \(M(1;0)=I(1)>0\), the function \(\mathfrak{B}\) is analytic in a neighborhood
of \(\nu=0\). Multiplying the boundary condition
\eqref{eq:zeta-bc-HH-proof} by \(\beta\), we see that the scalar eigenvalue
condition is exactly
\begin{equation}\label{eq:B-intersection-HH-proof}
    \mathfrak{B}(\nu)=\alpha-\beta\nu.
\end{equation}
Thus eigenvalues are intersections of the graph \(s=\mathfrak{B}(\nu)\) with the line
\(s=\alpha-\beta\nu\), and a double eigenvalue corresponds to a tangential
intersection.

At \(\nu=0\), we have \(M(\zeta;0)=I(\zeta)\). Hence
\begin{equation}\label{eq:B0-HH-proof}
    \mathfrak{B}(0)=\frac{a(1)^3I_\zeta(1)}{I(1)}
    =
    \frac1{I(1)}
    =
    \alpha_*.
\end{equation}
We also record the value of \(B'(0)\), since it identifies the
codimension-two point. Write
\[
    M(\zeta;\nu)=I(\zeta)+\nu M_1(\zeta)+O(\nu^2).
\]
Then \(M_1\) solves
\[
    \bigl(a^3(M_1)_\zeta\bigr)_\zeta=-aI,
    \qquad
    M_1(0)=0,
    \qquad
    (M_1)_\zeta(0)=0.
\]
Consequently,
\[
    a(\zeta)^3(M_1)_\zeta(\zeta)
    =
    -\int_0^\zeta a(t)I(t)\,\dd t,
\]
and
\[
    M_1(\zeta)=-\Theta(\zeta),
    \qquad
    \Theta(\zeta):=
    \int_0^\zeta a(r)^{-3}
    \int_0^r a(t)I(t)\,\dd t\,\dd r.
\]
By the quotient rule,
\begin{align*}
\mathfrak{B}'(0)
&=
\frac{
a(1)^3(M_1)_\zeta(1)I(1)
-
a(1)^3I_\zeta(1)M_1(1)
}
{I(1)^2}
\\
&=
\frac{
-I(1)\displaystyle\int_0^1 a(t)I(t)\,\dd t
+
\displaystyle\int_0^1 a(r)^{-3}
\int_0^r a(t)I(t)\,\dd t\,\dd r
}
{I(1)^2}.
\end{align*}
Using Fubini's theorem,
\[
\int_0^1 a(r)^{-3}
\int_0^r a(t)I(t)\,\dd t\,\dd r
=
I(1)\int_0^1 a(t)I(t)\,\dd t
-
\int_0^1 a(t)I(t)^2\,\dd t.
\]
Therefore, we obtain
\begin{equation}\label{eq:Bprime-HH-proof}
    \mathfrak{B}'(0)
    =
    -\frac{1}{I(1)^2}
    \int_0^1 a(t)I(t)^2\,\dd t
    =
    -\alpha_*^2
    \int_0^1 a(t)I(t)^2\,\dd t
    =
    -\beta_*.
\end{equation}
Thus the zero spectrum satisfies
\[
    \mathfrak{B}(0)=\alpha_*,
    \qquad
    \mathfrak{B}'(0)=-\beta_*,
\]
so the limiting tangency point is precisely \((\beta_*,\alpha_*)\). This is
the codimension-two point at which the scalar eigenvalue \(\nu=0\) is
double, equivalently the spatial-dynamics operator has a quadruple zero
eigenvalue.

We now parametrize the Hamiltonian--Hopf curve near this point. A
Hamiltonian--Hopf tangency with imaginary spatial eigenvalues
\[
    \kappa=\pm i\mathfrak q
\]
corresponds to
\[
    \nu=-\mathfrak q^2<0.
\]
The tangency conditions for \eqref{eq:B-intersection-HH-proof} are
\begin{equation}\label{eq:HH-tangency-conditions-proof}
    \mathfrak{B}(\nu)=\alpha-\beta\nu,
    \qquad
    \mathfrak{B}'(\nu)=-\beta.
\end{equation}
Substituting \(\nu=-\mathfrak q^2\) into
\eqref{eq:HH-tangency-conditions-proof} gives the exact local
parametrization
\begin{equation}\label{eq:HH-param-exact-proof}
    \beta_H(\mathfrak q)
    =
    -\mathfrak{B}'(-\mathfrak q^2),
    \qquad
    \alpha_H(\mathfrak q)
    =
    \mathfrak{B}(-\mathfrak q^2)
    +
    \mathfrak q^2\mathfrak{B}'(-\mathfrak q^2).
\end{equation}
Since \(\mathfrak{B}\) is analytic near \(0\), the functions \(\beta_H\) and
\(\alpha_H\) are smooth even functions of \(\mathfrak q\). Expanding
\eqref{eq:HH-param-exact-proof} and using
\eqref{eq:B0-HH-proof}--\eqref{eq:Bprime-HH-proof}, we obtain
\begin{equation}\label{eq:HH-param-expansion-proof}
    \beta_H(\mathfrak q)
    =
    \beta_*+\mathfrak{B}''(0)\mathfrak q^2
    +
    O(\mathfrak q^4),
    \qquad
    \alpha_H(\mathfrak q)
    =
    \alpha_*-\frac12\mathfrak{B}''(0)\mathfrak q^4
    +
    O(\mathfrak q^6).
\end{equation}
The nondegeneracy condition \(\mathfrak{B}''(0)<0\) ensures that this tangency curve
has a nonzero leading-order variation. Hence, the Hamiltonian--Hopf
resonance curve persists smoothly as \(\mathfrak q\to0\), and its limiting
point is precisely the codimension-two \(0^4\) point
\[
    (\beta_*,\alpha_*).
\]
This proves the smooth persistence assertion.

It remains to prove the separation from the weak-effective-gravity regime.
Assume that the laminar flow is uniformly separated from stagnation. Then
there exists \(a_{\min}>0\) such that
\[
    a(\zeta)\ge a_{\min}>0,
    \qquad
    0\le \zeta\le 1.
\]
Since
\[
    \int_0^1 a(\zeta)^{-1}\,\dd\zeta=H(1)=1,
\]
we have
\[
    a(\zeta)^{-3}
    \le
    a_{\min}^{-2}a(\zeta)^{-1}.
\]
Therefore,
\[
    \int_0^1 a(\zeta)^{-3}\,\dd\zeta
    \le
    a_{\min}^{-2}.
\]
It follows that
\begin{equation}\label{eq:alpha-star-positive-HH-proof}
    \alpha_*
    =
    \left(
    \int_0^1a(\zeta)^{-3}\,\dd\zeta
    \right)^{-1}
    \ge
    a_{\min}^2
    >
    0.
\end{equation}
By \eqref{eq:HH-param-expansion-proof},
\[
    \alpha_H(\mathfrak q)\to\alpha_*
    \qquad
    \text{as }
    \mathfrak q\to0.
\]
Hence there exists \(\mathfrak q_0>0\) such that
\[
    \alpha_H(\mathfrak q)
    \ge
    \frac12\alpha_*
    >
    0,
    \qquad
    0\le |\mathfrak q|<\mathfrak q_0.
\]
Thus, the local Hamiltonian--Hopf curve is bounded away from the
weak-effective-gravity threshold \(\alpha=0\). Returning to dimensional
variables,
\[
    g_*^H(\mathfrak q)
    =
    \frac{m^2}{d^3}\alpha_H(\mathfrak q),
\]
so in particular
\[
    g_*^H(0)
    =
    \frac{m^2}{d^3}\alpha_*
    >
    0.
\]
Therefore, for \(|\mathfrak q|<\mathfrak q_0\), the local
Hamiltonian--Hopf curve cannot enter the weak-effective-gravity regime
\[
    g_*=g-2\Omega c\approx0.
\]
This proves the separation assertion and completes the proof of
Theorem~\ref{thm:HH}.
\end{proof}
\begin{center}
\bibliographystyle{alpha}
\bibliography{REF.bib}
\end{center}

\end{document}